\documentclass[11pt]{amsart}
\usepackage{hyperref}
\usepackage{amsfonts}
\usepackage{amsmath}
\usepackage{amssymb}
\usepackage{amscd}
\usepackage{graphicx}
\usepackage{enumitem}
\usepackage{latexsym}
\usepackage{color}
\usepackage{comment}
\usepackage{epsfig}
\usepackage{amsopn}
\usepackage{xypic}
\usepackage{mathrsfs}
\usepackage{array}
\usepackage[usenames,dvipsnames]{pstricks}
\usepackage{epsfig}

\theoremstyle{plain}
\newtheorem{lemma}{Lemma}[section]
\newtheorem*{theorem*}{Theorem}
\newtheorem*{lemma*}{Lemma}
\newtheorem*{proposition*}{Proposition}
\newtheorem*{conjecture*}{Conjecture}
\newtheorem*{corollary*}{Corollary}
\newtheorem*{problem*}{Problem}
\newtheorem{theorem}[lemma]{Theorem}

\newtheorem{corollary}[lemma]{Corollary}
\newtheorem{proposition}[lemma]{Proposition}

\newtheorem{problem}[lemma]{Problem}

\theoremstyle{definition}
\newtheorem{definition}[lemma]{Definition}
\newtheorem{example}[lemma]{Example}
\newtheorem{remark}[lemma]{Remark}

\newtheorem{warning}[lemma]{Warning}

\newcommand{\fto}[1]{\stackrel{#1}{\to}}

\newcommand{\Z}{\mathbb{Z}}

\newcommand{\CC}{\mathbb{C}}
\newcommand{\QQ}{\mathbb{Q}}
\newcommand{\RR}{\mathbb{R}}

\newcommand{\OO}{\mathcal{O}}
\newcommand{\te}{\otimes}

\newcommand{\id}{\mathrm{id}}

\newcommand{\bv}{{\bf v}}
\newcommand{\bu}{{\bf u}}
\newcommand{\bk}{{\bf k}}

\newcommand{\bw}{{\bf w}}

\newcommand{\cE}{\mathcal{E}}

\newcommand{\ZZ}{\mathbb{Z}}
\newcommand{\leqpar}{\underset{{\scriptscriptstyle (}-{\scriptscriptstyle )}}{<}}

\renewcommand{\P}{\mathbb{P}}

\newcommand{\PP}{\mathbb{P}}

\DeclareMathOperator{\ch}{ch}

\DeclareMathOperator{\Hom}{Hom}
\DeclareMathOperator{\Mat}{Mat}

\DeclareMathOperator{\rk}{rk}

\DeclareMathOperator{\ord}{ord}

\DeclareMathOperator{\Ext}{Ext}
\DeclareMathOperator{\ext}{ext}

\DeclareMathOperator{\sHom}{\mathcal{H}\kern -.5pt\mathit{om}}
\DeclareMathOperator{\sTor}{\mathcal{T}\kern -1.5pt\mathit{or}}
\DeclareMathOperator{\GL}{GL}
\DeclareMathOperator{\SL}{SL}

\DeclareMathOperator{\udim}{\underline{\dim}}
\DeclareMathOperator{\edim}{edim}

\newcommand{\leqor}{\underset{{\scriptscriptstyle (}-{\scriptscriptstyle )}}{<}}

\begin{document}

\setcounter{topnumber}{1}

\date{\today}
\author[I. Coskun]{Izzet Coskun}
\address{Department of Mathematics, Statistics and CS \\University of Illinois at Chicago, Chicago, IL 60607}
\email{coskun@math.uic.edu}

\author[J. Huizenga]{Jack Huizenga}
\address{Department of Mathematics, The Pennsylvania State University, University Park, PA 16802}
\email{huizenga@psu.edu}

\author[J. Kopper]{John Kopper}
\address{Department of Mathematics, The Pennsylvania State University, University Park, PA 16802}
\email{kopper@psu.edu}

\thanks{During the preparation of this article the first author was partially supported by the NSF FRG grant  DMS 1664296 and  the second author was partially supported  by NSF FRG grant DMS 1664303}

\title[Cohomology of tensor products of vector bundles on $\P^2$]{The cohomology of general tensor products of vector bundles on $\P^2$}

\begin{abstract}
Computing the cohomology of the tensor product of two vector bundles is central in the study of their moduli spaces and in applications to representation theory, combinatorics and physics. These computations play a fundamental role in the construction of Brill--Noether loci, birational geometry and $S$-duality. Using recent advances in the Minimal Model Program for moduli spaces of sheaves on $\PP^2$, we compute the cohomology of the tensor product of general semistable bundles on $\PP^2$.  This solves a natural higher rank generalization of  the polynomial interpolation problem. 
More precisely, let $\bv$ and $\bw$ be two Chern characters of stable bundles on $\PP^2$ and assume that $\bw$ is sufficiently divisible depending on $\bv$. Let $V \in M(\bv)$ and $W \in M(\bw)$ be two general stable bundles. We fully compute the cohomology of $V \otimes W$.  In particular, we show that if $W$ is exceptional, then $V \otimes W$ has at most one nonzero cohomology group determined by the slope and the Euler characteristic, generalizing foundational results of Dr\'{e}zet, G\"{o}ttsche and Hirschowitz. We characterize the invariants of effective Brill--Noether divisors on $M(\bv)$. We also characterize when $V\te W$ is globally generated. Our computation is canonical given the birational geometry of the moduli space, suggesting a roadmap for tackling analogous problems on other surfaces.
\end{abstract}

\maketitle

\setcounter{tocdepth}{1}
\tableofcontents

\section{Introduction}

Computing the cohomology of the tensor product of two vector bundles is central in the study of their moduli spaces and in applications to representation theory, combinatorics and physics. These computations play a fundamental role in the construction of Brill--Noether loci, birational geometry and $S$-duality. In the last decade, thanks to advances in Bridgeland stability and the Minimal Model Program, there has been rapid development in our understanding of the birational geometry of moduli spaces of sheaves. In this paper we use these advances to compute the cohomology of the tensor product of general semistable bundles on $\PP^2$.  Our computation is canonical given the birational geometry of the moduli space, suggesting a roadmap for tackling analogous problems on other surfaces. The Fano fibration structure on the moduli space yields a canonical resolution of the general sheaf that makes our computation possible.

The higher rank interpolation problem underlies the construction and geometric properties of Brill--Noether divisors. It is fundamental for computing the ample and effective cones of moduli spaces of sheaves,  the study of the birational geometry of moduli spaces of sheaves and the $S$-duality conjecture (see \cite{ABCH, CoskunHuizengaGokova, CHW}). Let  $M(\bv)$ denote the moduli space of Gieseker semistable sheaves on $\PP^2$ with Chern character $\bv$. 

\begin{problem}[Generic Higher Rank  Interpolation]\label{prob-HRI}
Let $\bv$ and $\bw$ be the Chern characters of semistable sheaves on $\PP^2$. Let $V \in M(\bv)$ and $W \in M(\bw)$ be  general semistable sheaves. Compute the cohomology of $V \otimes W$.
\end{problem}

When $\bv$ has rank 1, the generic interpolation problem reduces to the classical question of when  general points impose  independent conditions on sections of a vector bundle---a problem that has been at the heart of algebraic geometry since the inception of the field (see \cite{CoskunHuizengaMonomial}). Problem \ref{prob-HRI} generalizes the classical interpolation problem to higher ranks and is fundamental to the study of higher rank moduli spaces.

Let $\bv$ and $\bw$ be two Chern characters of semistable sheaves on $\PP^2$ and assume that $\bw$ is sufficiently divisible, depending on $\bv$. Let $V \in M(\bv)$ and $W \in M(\bw)$ be two general semistable sheaves. In this paper, we compute the cohomology of $V \otimes W$ and solve the generic higher rank interpolation problem.  In particular, we show that if $W$ is exceptional and the rank of $V$ is at least 2, then $V \otimes W$ has at most one nonzero cohomology group, determined by the slope and the Euler characteristic. This generalizes foundational results of Dr\'{e}zet \cite{DrezetBeilinson, Drezet} and G\"{o}ttsche and Hirschowitz \cite{GottscheHirschowitz}. We also characterize the invariants of effective Brill--Noether divisors on $M(\bv)$ and determine when $\sHom(W,V)$ is globally generated. 

Recent developments in our understanding of the birational geometry of $M(\bv)$ provide the key inputs to our calculation of the cohomology groups of $V \otimes W$. The final model in the Minimal Model Program for $M(\bv)$ is given by a Fano fibration to a Kronecker moduli space. This fibration yields a canonical resolution of the general sheaf in $M(\bv)$. This resolution enables us to carry out the computation. We will now explain our results and methods in greater detail.

\subsection{The Kronecker fibration and the main theorem} The \emph{Kronecker fibration} of a positive dimensional moduli space $M(\bv)$ is a dominant rational map $M(\bv)\dashrightarrow Kr(\bv)$ to a moduli space $Kr(\bv)$ of semistable representations of a Kronecker quiver constructed from $\bv$.  To describe this map, we need the concept of the \emph{(primary) corresponding exceptional bundle $E_+$ to $\bv$}.  The bundle $E_+$ is the exceptional bundle of smallest slope with the property that if $G$ is any exceptional bundle satisfying $\mu(G) > \mu(E_+)$, then $\chi(V\te G) >0$.  However, the number $\chi(V\te E_+)$ can be positive, zero, or negative.

For simplicity in the introduction, let us focus on the case where $\chi(V\te E_+) > 0$.  If $V\in M(\bv)$ is a general sheaf, then $\Hom(E_+^*,V)$ has the expected dimension $\chi(V\te E_+)$ and we can consider the mapping cone $K$ of the canonical evaluation $$E_+^* \te \Hom(E_+^*,V) \to V \to K\to \cdot.$$ A Beilinson spectral sequence shows that there is a pair of exceptional bundles $F,G$ and exponents $m_1,m_2$ such that $K$ is isomorphic in the derived category to a two-term complex $$K: F^{m_1} \to  G^{m_2}$$ sitting in degrees $-1$ and $0$.  The linear-algebraic data of a map $F^{m_1}\to G^{m_2}$ can be encoded as a representation of the Kronecker quiver with two vertices and $N = \dim \Hom(F,G)$ arrows.  The rational map $M(\bv) \dashrightarrow Kr(\bv)$ is defined by mapping $V\mapsto K$.  The Chern character $$\bk:= \ch K = \bv - \chi(V\te E_+)\ch E_+^*$$ depends only on $\bv$.  We now state our main theorem.

\begin{theorem}\label{thm-main1}
Let $\bv,\bw$ be Chern characters of stable bundles on $\P^2$.  Suppose the moduli space $M(\bv)$ is positive dimensional and that $\bw$ is sufficiently divisible (depending on $\bv$).  Let $V\in M(\bv)$ and $W\in M(\bw)$ be general bundles.
\begin{enumerate}
\item If $\chi(\bv \te E_+)\leq 0$, then either $H^0(V\te W)=0$ or $H^1(V\te W) =0$.
\item If $\chi(\bv \te E_+) > 0$ and $\rk(\bk) \leq 0$, then either $H^0(V\te W) = 0$ or $H^1(V\te W) = 0$.
\item Suppose $\chi(\bv\te E_+) > 0$ and $\rk(\bk) > 0$. Then $H^0(V\te W)$ and $H^1(V\te W)$ can be computed as follows.
\begin{enumerate}
\item If $\chi(\bk \te \bw) \geq 0$ or $\chi(\bw \te E_+^*) \leq 0$, then either $H^0(V\te W)=0$ or $H^1(V\te W)=0$.

\item Otherwise, $V\te W$ is special and 
\begin{align*}h^0(V\te W) &= \chi(\bv\te E_+)\chi(\bw\te E_+^*)\\
h^1(V\te W) &= -\chi(\bk \te \bw).
\end{align*}
\end{enumerate}
\end{enumerate}
\end{theorem}

By stability, we can only have $H^0(V\te W) \neq 0$ if $\mu(V\te W)\geq 0$.  Similarly, by Serre duality and stability, we can only have $H^2(V\te W) \neq 0$ if $\mu(V\te W) \leq -3$.  Therefore if we apply  Theorem \ref{thm-main1} to the Serre dual character $\bv^D$, we can determine all the cohomology of $V\te W$ in every case.  

\begin{remark}
For applications of Theorem \ref{thm-main1} to birational geometry that depend on asymptotic linear series, the assumption that $\bw$ is sufficiently divisible does not matter.   Moreover, the proof of Theorem \ref{thm-main1} is sufficiently explicit that the precise multiple of $\bw$ that is needed is easily determined.  In particular, if $\chi(\bk \te \bw) \geq 0$ and $\chi(\bw \te E_+^*) \geq 0$ then no multiple is needed at all; this applies to most characters $\bw$ where $\chi(\bv \te \bw) > 0$.  Additionally, no multiple of $\bw$ is needed in case (3b) of the theorem, where there is special cohomology.  In cases where $\chi(\bv\te \bw) \leq 0$, a multiple is often needed, but the appropriate multiple is easy to compute.  For more details, see Section \ref{sec-overview}.
\end{remark}

\begin{figure}[t] 
\begin{center}
\setlength{\unitlength}{1in}
\begin{picture}(4.584,2.31)
\put(0,0){\includegraphics[scale=.6,bb=0 0 7.64in 3.85in]{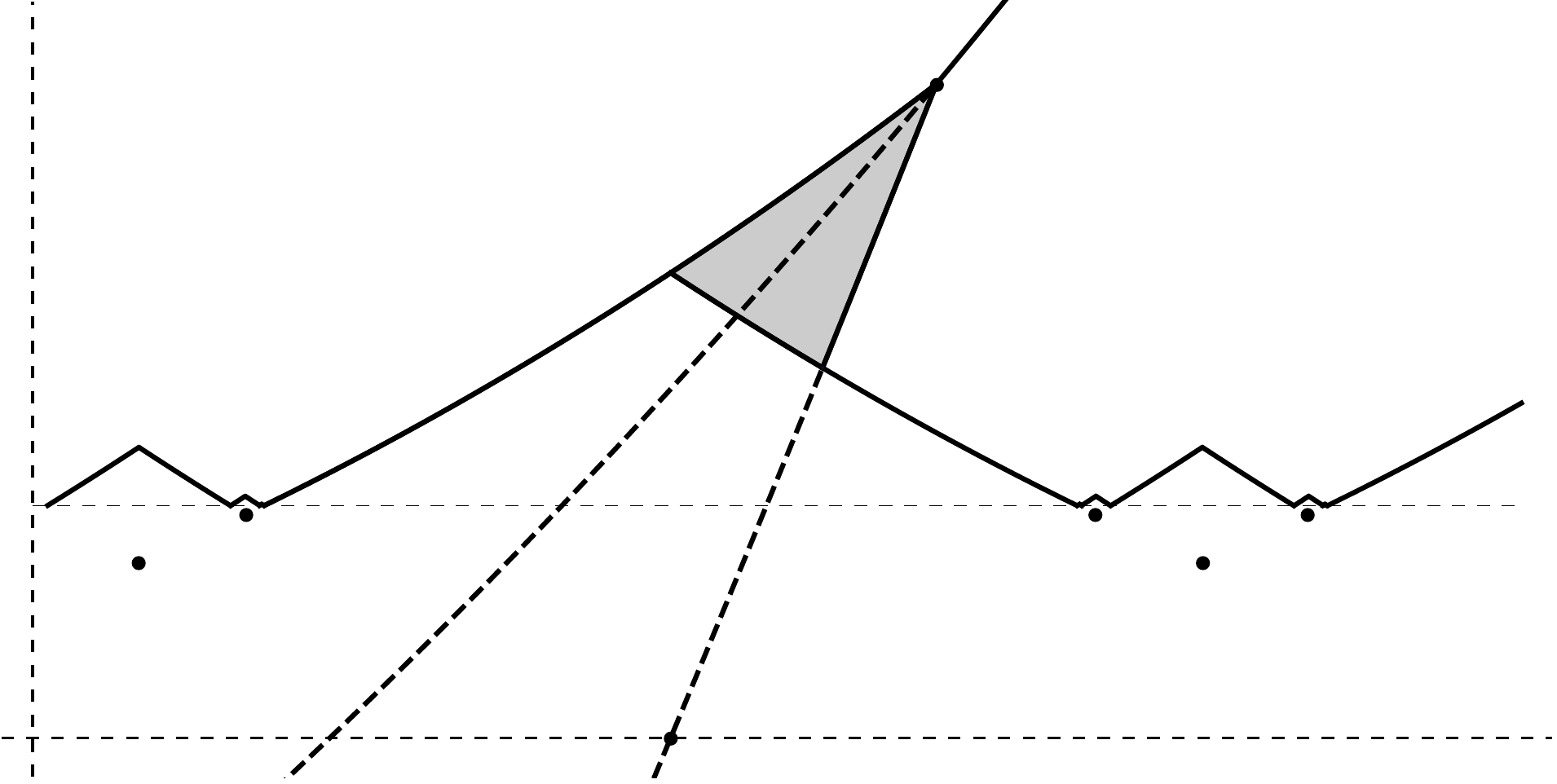}}
\put(1,1.15){$(E_{+}^*)^\perp$}
\put(1.75,.2){$E_+$}
\put(2.7,2.2){$\bv^\perp$}
\put(2.8,2){$\bu^+$}
\put(2.2,.5){$\bk^\perp$}
\put(3.1,1.5){$H^1(V\te W) = 0$}
\put(.4,1.5){$H^0(V\te W) = 0$}
\put(-.05,.78){$\frac{1}{2}$}
\put(3.75,1.1){$\Delta = \delta(\mu)$}
\put(.13,2.1){$\Delta$}
\put(4.4,.2){$\mu$}
\end{picture}
\end{center}
\caption{Illustration of Theorem \ref{thm-main1} in case (3).  The character $\bv$ is fixed and $\bw = (r,\mu,\Delta)$ is variable. In the interior of the shaded region, both $H^0(V\te W)$ and $H^1(V\te W)$ are nonzero.  See Example \ref{ex-intro}. The picture is to scale for the character $\bv =(4,1,9/4)$ with $E_+ = \OO_{\P^2}$.  }\label{fig-intro}
\end{figure}
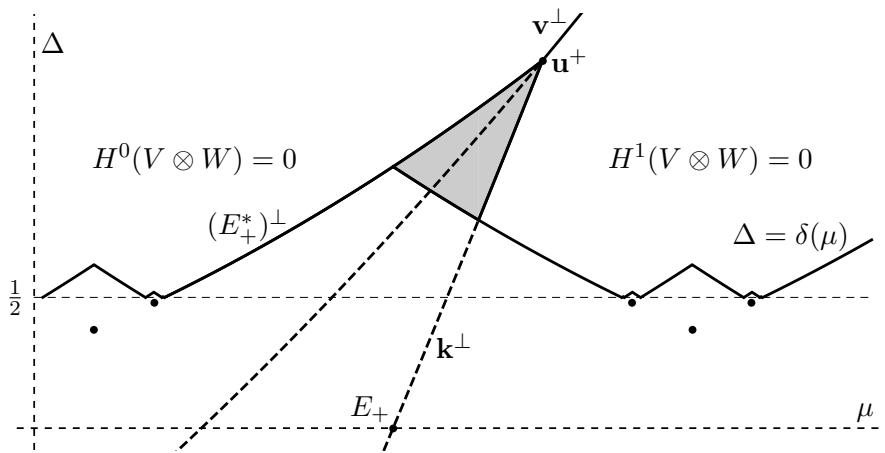

\begin{example}\label{ex-intro}
We illustrate Theorem \ref{thm-main1} in case (3) in Figure \ref{fig-intro}. We view the Chern character $\bv$ as fixed and $\bw = (r,\mu,\Delta)$ as variable, where $\mu$ and $\Delta$ are the slope and discriminant (see \S\ref{ss-basic}).  The coordinates on the plane are $(\mu,\Delta)$.  By abuse of notation we plot a Chern character of slope $\mu$ and discriminant $\Delta$ at the point $(\mu,\Delta)$.  The nondegenerate symmetric pairing $\chi(-\te-)$ on $K(\P^2)$ defines orthogonal complements of Chern characters such as $\bv^\perp \subset K(\P^2)$.  In slope and discriminant coordinates, these subspaces become parabolas (see \S\ref{sss-orthogonalParabolas}).   We have, for example, $\chi(\bv \te \bw) \leq 0$ when $(\mu,\Delta)$ lies above the parabola labeled by $\bv^\perp$.

There is a fractal-like curve $\Delta = \delta(\mu)$ in the $(\mu,\Delta)$-plane, the \emph{Dr\'ezet--Le Potier curve}, such that $\bw$ is stable with positive dimensional moduli space if and only if $(\mu,\Delta)$ lies on or above the curve (see \S\ref{sss-DLP}).  Since we are assuming $\chi(\bv \te E_+)>0$, the parabola $\bv^\perp$ lies above the invariants of $E_+$.  The parabola $\bk^\perp$ passes through $E_+$, and the assumption that $\rk(\bk) >0$ implies this parabola is increasing near $E_+$.  We shade the region where $\bw$ is stable, $\chi(\bk\te \bw) < 0$, and $\chi(\bw\te E_+^\ast) > 0$.  In this region, both $H^0(V\te W)$ and $H^1(V\te W)$ are nonzero.  Outside of this region, consideration of $\chi(\bv\te \bw)$ allows us to determine whether $H^0(V\te W) = 0$ or $H^1(V\te W) = 0$.
\end{example}

In exchange for some precision, a cleaner version of Theorem \ref{thm-main1} can be stated when the discriminants of $\bv$ and $\bw$ are sufficiently large.  

\begin{corollary}
Let $\bv,\bw$ be Chern characters of stable bundles on $\P^2$ with $\Delta(\bv) \geq 3$ and $\Delta(\bw) \geq 3$.  Suppose that $\bw$ is sufficiently divisible (depending on $\bv$).  If $V\in M(\bv)$ and $W\in M(\bw)$ are general bundles then $V\te W$ has at most one nonzero cohomology group.
\end{corollary}

The proof is a straightforward computation using Theorem \ref{thm-main1}.  The result is easily seen to be false if the constant $3$ is replaced by any smaller number: for $\bv = (r,\mu,\Delta) = (2,0,3)$, the character $\bu^+$ in Figure \ref{fig-intro} has discriminant $3$ and the shaded region contains characters with discriminant arbitrarily close to $3$.

We now discuss some of the most interesting special cases of Theorem \ref{thm-main1}.  Each of these cases will give important steps in the full proof of the theorem.  A more detailed roadmap to the proof is provided in \S\ref{sec-overview}.

\subsection{Exceptional bundles} Theorem \ref{thm-main1} in particular applies when $W=E$ is an exceptional bundle.  A tensor product $V\te E$ is never special.

\begin{theorem}\label{thm-introExc}
Let $\bv$ be the Chern character of a stable bundle on $\P^2$, and let $E$ be an exceptional bundle.  If $V\in M(\bv)$ is general, then $V\te E$ has at most one nonzero cohomology group.
\end{theorem}

We prove this theorem in \S\ref{sec-exceptional}; the proof depends on Dr\'ezet's inductive description of exceptional bundles and the Kronecker fibration.  Since we often use exceptional bundles as building blocks for other stable bundles, Theorem \ref{thm-introExc} is the starting point for many of our later computations.

Special cases of Theorem \ref{thm-introExc} have played a foundational role in the study of the moduli spaces $M(\bv)$. For example, when $E= \OO_{\PP^2}$, we recover a theorem of G\"{o}ttsche and Hirschowitz asserting that the general semistable bundle of rank at least 2 has at most one nonzero cohomology group \cite{GottscheHirschowitz, CoskunHuizengaBN}. When $V$ is an exceptional bundle, we recover a theorem of Dr\'{e}zet that plays a crucial role in computations with Beilinson spectral sequences \cite{DrezetBeilinson, Drezet}.

\subsection{Bundles with dual corresponding exceptional bundle}

One of the most interesting cases of Theorem \ref{thm-main1} occurs when $\bw$ is a character such that $M(\bw)$ has positive dimension and the primary corresponding exceptional bundles to $\bv$ and $\bw$ are dual.  In particular, these hypotheses are satisfied in case (3b) of the theorem.  In this case, using Serre duality, we associate a two-term complex $K':F^{n_1}\to G^{n_2}$ to the general sheaf $W\in M(\bw)$.  The computation of the cohomology of $V\te W$ reduces to the computation of the groups $\Ext^i(K',K)$. 

A key feature of the complex $K$ is that it corresponds to a \emph{stable} representation of the Kronecker quiver.  We then prove the following result regarding Kronecker quivers. This refines a result of Schofield \cite{Schofield} that holds for arbitrary quivers.

\begin{theorem}
Let $K$ and $K'$ be general representations of the $N$-arrowed Kronecker quiver, and assume one of them is semistable.  Then at most one of the groups $\Ext^i(K',K)$ is nonzero.
\end{theorem}

Historically, many different types of resolutions have been used to study properties of a general sheaf $V$.  For example, in their classification of semistable sheaves on $\P^2$, Dr\'ezet and Le Potier made use of general resolutions of the form $$0\to \OO_{\P^2}(-1)^{m_1}\to T_{\P^2}(-2)^{m_2} \oplus \OO_{\P^2}^{m_3} \to V(a)\to 0,$$ where $m_3 = \chi(V(a))$  and the map $\OO_{\P^2}^{m_3}\to V(a)$ is the canonical evaluation.  The mapping cone of this evaluation can be identified with a complex $\OO_{\P^2}(-1)^{m_1}\to T_{\P^2}(-2)^{m_2}$; however, this complex often corresponds to an unstable representation of the Kronecker quiver. In long exact sequences for cohomology arising from this resolution, there will usually be a map whose rank is hard to calculate.  This resolution therefore is not a good tool for cohomology computations.

\subsection{Bundles with orthogonal cohomology}

The other main interesting case of Theorem \ref{thm-main1} occurs when $\chi(\bv\te \bw) = 0$.
In this case, if the general tensor product $V\te W$ has no cohomology, then there are induced Brill--Noether divisors on the moduli spaces $M(\bv)$ and $M(\bw)$.  Furthermore, these moduli spaces are candidate pairs for Le Potier's strange duality.  

Example \ref{ex-intro} shows that there exist stable characters $\bv$ and $\bw$ satisfying $\chi(\bv\te \bw) =0$, yet the general tensor product $V\te W$ has nonzero cohomology. Indeed, $\bv^\perp$ passes through the interior of the shaded region in Figure \ref{fig-intro}.  On the other hand, in \cite{CHW} two characters $\bu^\pm$ were constructed with the property that if $U^{\pm} \in M(\bu^{\pm})$ are general, then $V\te U^{\pm}$ has no nonzero cohomology groups.  The character $\bu^+$ is depicted in Figure 1 (see Definition \ref{def-corrOrth} for the general definition).  These bundles define Brill--Noether divisors  $$\{V\in M(\bv):h^0(V\te U^+)>0\}\subset M(\bv)$$ and $$\{V\in M(\bv):h^2(V\te U^-)>0\}\subset M(\bv).$$  When $\rk(\bv)\geq 3$ and $M(\bv)$ has Picard rank 2, these divisors were found to be the extremal edges of the movable cone of divisors for the moduli space $M(\bv)$ (see \cite{LiZhao}).

Here we use the bundles $U^{\pm}$ as a starting point to construct new stable bundles $W$ such that $V\te W$ has no cohomology.  We prove the following theorem.

\begin{theorem}\label{thm-orthogonalIntro}
Let $\bv,\bw$ be Chern characters of stable bundles on $\P^2$.  Suppose $M(\bv)$ is positive dimensional and $\bw$ is sufficiently divisible (depending on $\bv$).    Assume  $\chi(\bv\te \bw) = 0$ and either 
\begin{enumerate}
\item $\mu(\bw)\geq \mu(\bu^+)$ or
\item $\mu(\bw)\leq \mu(\bu^-)$.
\end{enumerate}
Then if $V\in M(\bv)$ and $W\in M(\bw)$ are general, $V\te W$ has no cohomology.
\end{theorem}

A key step in the proof is to develop a criterion for a general tensor product $V\te W$ to be globally generated. This is a quite powerful tool in its own right.  

\begin{theorem}\label{thm-ggIntro}
Let $\bv,\bw$ be Chern characters of stable bundles on $\P^2$ such that $M(\bv)$ and $M(\bw)$ are positive dimensional.  Let $E_+$ be the primary corresponding exceptional bundle to $\bv$ and let $E_+'$ be the primary corresponding exceptional bundle to $\bw$.   If $$\mu(E_+)+ \mu(E_+') \leq -1,$$ then $V\te W$ is globally generated for general $V\in M(\bv)$ and $W\in M(\bw)$.
\end{theorem}

Theorem \ref{thm-ggIntro} can be equivalently stated in terms of the global generation of $\sHom(W,V)$ (see Theorem \ref{thm-gghom}). This allows the use of Bertini-type theorems to construct bundles with interesting cohomology with applications to higher codimension Brill--Noether loci. 

\subsection{Hilbert schemes of points}  While all the previous results were stated for characters $\bv$ of stable \emph{bundles}, most of the proof of Theorem \ref{thm-main1} holds verbatim if instead $\bv = (r,\mu,\Delta) = (1,a,n)$ is a rank 1 character and $M(\bv)$ is the moduli space parameterizing twisted ideal sheaves $I_Z(a)$. The only step that does not immediately generalize is Theorem \ref{thm-orthogonalIntro}, but a suitable analog of Theorem \ref{thm-orthogonalIntro} was proved in \cite[Theorem 3.8]{CoskunHuizengaMonomial} for rank 1 sheaves.  Also, when $\chi(V\te E_{\nu^+})>0$ the complex $K$ always has nonpositive rank.  Thus in this case the following cleaner statement holds.

\begin{theorem}
Let $\P^{2[n]}$ be the Hilbert scheme parameterizing ideal sheaves $I_Z$ of zero-dimensional schemes $Z\subset \P^2$ of length $n$.  Let $\bw$ be a stable character with $\mu(\bw) \geq 0$ and suppose $\bw$ is sufficiently divisible (depending on $n$).  If $I_Z\in \P^{2[n]}$ and $W\in M(\bw)$ are general, then $I_Z\te W$ has at most one nonzero cohomology group.
\end{theorem}

On the other hand, $H^2(I_Z\te W)\cong H^2(W)$ is easily computed by Serre duality, but is often larger than the Euler characteristic would predict.

\subsection{Torsion sheaves}
Our main result can also be refined to allow for moduli spaces of one-dimensional semistable sheaves.  Here the cohomology is always as expected.  See \S\ref{ssec-torsion} for details.

\begin{theorem}
Let $\bv$ be the Chern character of a one-dimensional semistable sheaf, and let $\bw$ be the Chern character of a stable bundle.  Suppose that $\bw$ is sufficiently divisible.  If $V\in M(\bv)$ and $W\in M(\bw)$ are general sheaves, then $V\te W$ has at most one nonzero cohomology group.
\end{theorem}

\subsection*{Organization of the paper}
In \S\ref{sec-prelim} we recall basic facts on moduli of sheaves, exceptional bundles, the Dr\'ezet-Le Potier classification of stable bundles, and the Kronecker fibration.  We study homomorphisms between representations of Kronecker quivers in \S\ref{sec-kronecker}.  Next we turn to studying the cohomology of a tensor product $V\te E$ with $E$ an exceptional bundle in \S\ref{sec-exceptional}.

We give an overview of the proof of Theorem \ref{thm-main1} in \S\ref{sec-overview}.  We begin the proof of the main theorem in \S\ref{sec-regions}, and in particular focus on the case where the corresponding exceptional bundles to $V$ and $W$ are dual.  In \S\ref{sec-orthogonal} we study the case where $\chi(\bv\te \bw)=0$.  Finally, in \S\ref{sec-interpolation} we compute the cohomology of $V\te W$ in the remaining cases.  

\subsection*{Acknowledgements} We would like to thank Arend Bayer, Aaron Bertram, Lawrence Ein, Joe Harris, Emanuele Macr\`{i}, Sam Shideler, and Jonathan Wolf for valuable conversations related to the subject matter of the paper.  We would also like to thank the referee who provided several helpful comments on the paper.

\section{Preliminaries}\label{sec-prelim}

In this section, we recall basic facts concerning semistable sheaves on $\PP^2$. We refer the reader to  \cite{CoskunHuizengaGokova, CHW} and \cite{LePotier} for more details.

\subsection{Basic definitions}\label{ss-basic} Every sheaf in this paper will be a torsion-free  coherent sheaf on $\PP^2$ unless explicitly specified otherwise.  Let $K(\P^2) \cong \Z^3$ be the $K$-group.  We will write Chern characters $\bv\in K(\P^2)$ of positive rank as triples $\bv = (r,\mu,\Delta)$, where we define the slope $\mu$ and the discriminant $\Delta$ by 
$$\mu = \frac{\ch_1}{r} \quad\textrm{and}\quad \Delta = \frac{\mu^2}{2} - \frac{\ch_2}{r}.$$ The slope and the discriminant of a sheaf is defined as the slope and the discriminant of its Chern character, respectively. The advantage of the slope and the discriminant is that they are additive on tensor products: 
$$\mu(V \otimes W) = \mu(V) + \mu(W) \quad\textrm{and} \quad \Delta(V \otimes W) = \Delta(V) + \Delta(W).$$ In particular, taking $W = \OO_{\P^2}^{\oplus k}$, we have $\mu(V^{\oplus k}) = \mu(V)$ and $\Delta(V^{\oplus k}) = \Delta(V)$.   In terms of these invariants, the Riemann-Roch formula reads
$$\chi(V) = r(V) ( P(\mu(V)) - \Delta(V)), $$ where $$P(x) = \frac{1}{2} x^2 + \frac{3}{2} x + 1$$ is the Hilbert polynomial of $\OO_{\PP^2}$. 
In particular, 
$$\chi(V, W) = \sum_{i=0}^2 (-1)^i \ext^i(V, W) = r(V) r(W) \left(P\left(\mu(W) - \mu(V)\right) - \Delta(V) - \Delta(W)\right),$$ where $\ext^i(V,W)$ denotes the dimension of $\Ext^i(V,W)$. 

The Hilbert polynomial $P_V$ and the reduced Hilbert polynomial $p_V$ of a torsion-free sheaf are defined by
$$P_V(m) = \chi(V(m))= r(V) \frac{m^2}{2} + \mbox{l.o.t.}  \quad\textrm{and} \quad p_V = \frac{P_V}{r(V)}.$$
A sheaf $V$ is {\em  (semi)-stable} if $p_W(m) \leqpar p_V(m)$ for every proper subsheaf $W \subset V$ and $m\gg 0$. 

Every sheaf admits a unique {\em Harder--Narasimhan filtration} such that the successive quotients are semistable. Furthermore, every semistable sheaf admits a {\em Jordan--H\"{o}lder filtration} into stable sheaves. Two semistable sheaves are called {\em $S$-equivalent} if they have the same associated graded object with respect to the Jordan--H\"{o}lder filtration. There exists a projective moduli space $M(\bv)$ parameterizing $S$-equivalence classes of  semistable sheaves on $\PP^2$ with Chern character $\bv$ \cite{Gieseker, Maruyama}.  

\subsection{The classification of stable bundles on $\PP^2$} We now recall Dr\'{e}zet and Le Potier's classification of Chern characters of stable bundles on $\PP^2$. We refer the reader to \cite{DLP, LePotier, CoskunHuizengaGokova} for further details.

\subsubsection{Exceptional bundles} An {\em exceptional bundle} $E$ on $\PP^2$ is a stable bundle such that $\Ext^1(E,E) = 0$.  An {\em exceptional slope } $\alpha$ is the slope of an exceptional bundle. We denote the set of exceptional slopes by $\mathcal{E}$. Given an exceptional slope $\alpha \in \mathcal{E}$, there is a unique exceptional bundle $E_{\alpha}$ with that slope. For an exceptional bundle $E$, we have $\chi(E,E)=1$, hence the rank of $E$ and $\ch_1(E)$ are relatively prime.  Consequently, the rank of the exceptional bundle with slope $\alpha$ is the smallest positive integer $r_{\alpha}$ such that $r_{\alpha} \alpha$ is an integer. By Riemann--Roch, the discriminant of $E_\alpha$ is given by 
$$\Delta_{\alpha} = \frac{1}{2} \left(1 - \frac{1}{r_{\alpha}^2} \right).$$
The exceptional bundles are the stable bundles $E$ on $\PP^2$ with $\Delta(E) < \frac{1}{2}$. They are rigid and their moduli spaces consist of a single reduced point.

Dr\'{e}zet has given a complete classification of exceptional bundles on $\PP^2$ \cite{Drezet}. Line bundles are exceptional. Every other exceptional bundle on $\PP^2$ can be obtained from line bundles by a sequence of mutations. This description also yields an explicit one-to-one correspondence $\varepsilon: \ZZ[\frac{1}{2}] \to \mathcal{E}$ between dyadic integers and the exceptional slopes, defined inductively by 
$\varepsilon(n) = n$ for an integer $n$ and 
$$\varepsilon \left( \frac{2p+1}{2^{q+1}} \right) = \varepsilon \left( \frac{p}{2^q} \right) .  \varepsilon \left( \frac{p+1}{2^q} \right),$$ where $$\alpha . \beta = \frac{\alpha+\beta}{2} + \frac{\Delta_{\beta} - \Delta_{\alpha}}{3+\alpha-\beta}.$$
Thus every exceptional slope $\nu\in \cE$ can be written uniquely as $\nu = \alpha.\beta$ where $\alpha$ and $\beta$ are of the form $\alpha = \varepsilon(p/2^q)$ and $\beta = \varepsilon((p+1)/2^q)$.  The {\em order} of an exceptional bundle of slope $\alpha$ is the smallest integer $q\geq 0$ such that $\alpha = \varepsilon\left(p/2^q \right)$. In many arguments, we will induct on the order of an exceptional bundle.

\subsubsection{Higher dimensional moduli spaces}\label{sss-DLP} Each exceptional bundle provides an obstruction for the existence of semistable sheaves. These obstructions can be efficiently described by a fractal-like curve $\delta$ in the $(\mu, \Delta)$-plane called the {\em Dr\'{e}zet--Le Potier curve}. Explicitly, define
$$\delta(\mu)= \sup_{\alpha \in \mathcal{E}: | \mu- \alpha| < 3 } ( P(- |\mu - \alpha|) - \Delta_{\alpha}).$$ The following theorem of Dr\'{e}zet and Le Potier gives the classification of Chern characters of semistable bundles on $\PP^2$. 

\begin{theorem}\cite{DLP}\label{thm-DLP}
Let $\bv = (r, \mu, \Delta)$ be an integral Chern character of positive rank. Then the moduli space $M(\bv)$ is positive dimensional if and only if $\Delta \geq \delta(\mu)$.  If $r\geq 2$, then the general sheaf $V\in M(\bv)$ is a vector bundle.
\end{theorem}

We depict the graph of the Dr\'ezet--Le Potier function and shade the region of Chern characters of stable bundles with positive dimensional moduli spaces in Figure \ref{fig-DLP}.

\begin{figure}[t]

\begin{center}
\setlength{\unitlength}{1in}
\begin{picture}(5.34,2.15)
 \put(0,0){\includegraphics[scale=.5,bb=0 0 10.69in 4.31in]{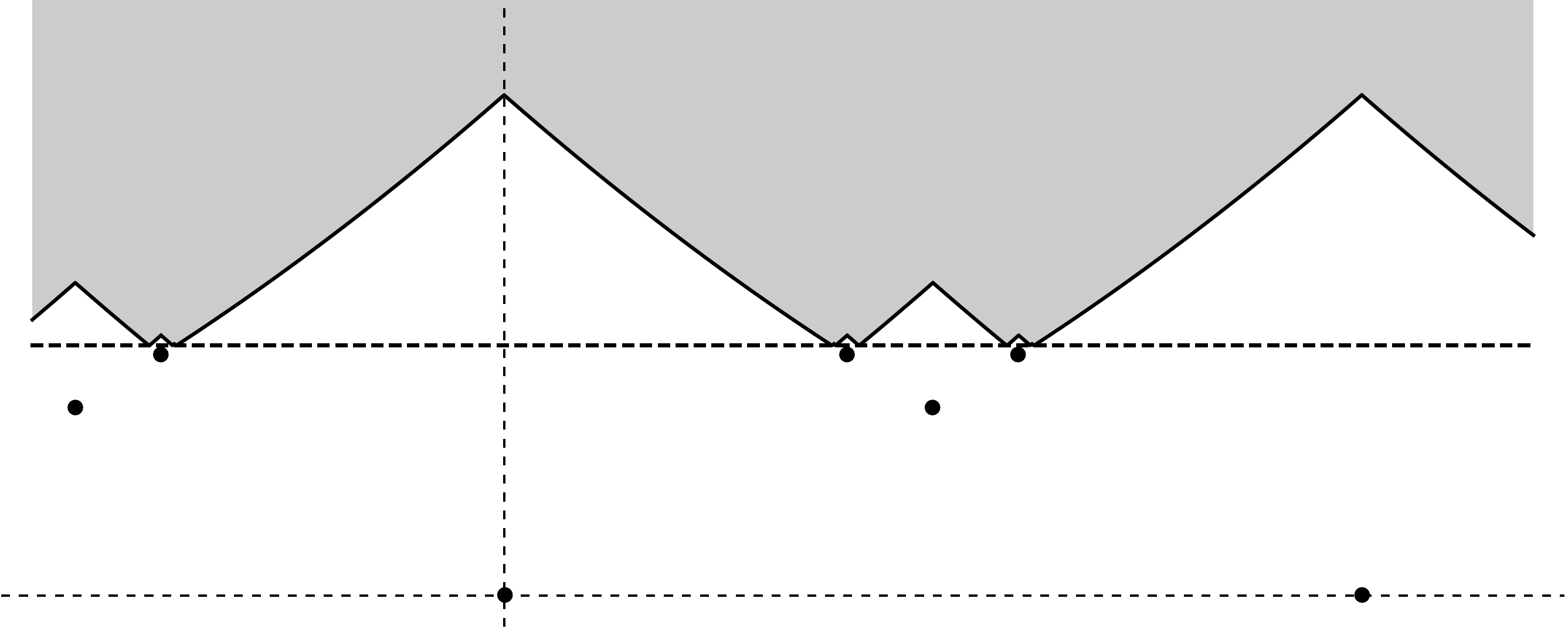}}
\put(5,.2){$\mu$}
\put(1.75,.2){$\OO_{\P^2}$}
\put(4.2,.2){$\OO_{\P^2}(1)$}
\put(2.9,.59){$T_{\P^2}(-1)$}
\put(1.75,2){$\Delta$}
\put(1.75,1.8){$1$}
\put(1.75,1.05){$1/2$}
\put(1.75,.0){$0$}
\put(4.65,.0){$1$}
\put(3.5,1.45){$\Delta=\delta(\mu)$}
\end{picture}
\end{center}
\caption{The Dr\'ezet--Le Potier curve.  Chern characters of stable bundles with positive dimensional moduli spaces lie in the shaded region above the curve.  The first several exceptional bundles are also displayed.  See Theorem \ref{thm-DLP}.}\label{fig-DLP}
\end{figure}
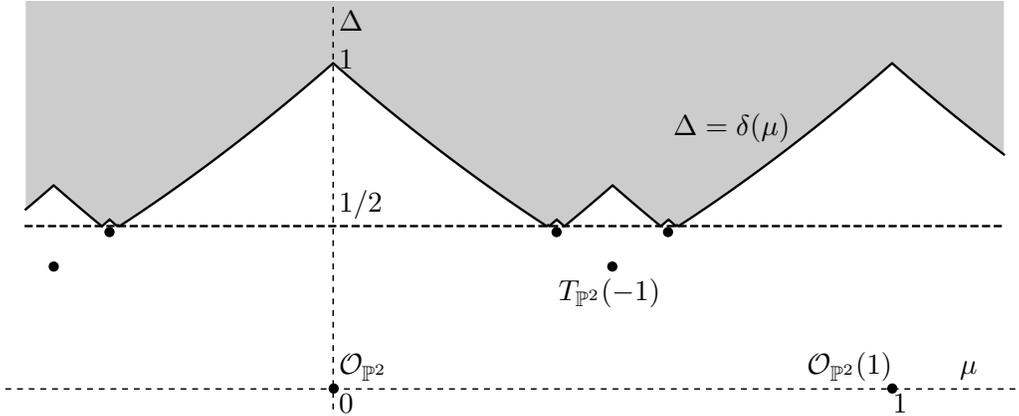

For each exceptional slope $\alpha \in \mathcal{E}$, there exists an interval 
$I_{\alpha} = ( \alpha - x_{\alpha}, \alpha+x_{\alpha})$  with $$x_{\alpha} =\frac{3- \sqrt{5+8\Delta_{\alpha}}}{2}$$ such that if $\mu\in I_\alpha$, then the supremum defining $\delta(\mu)$ is achieved by $\alpha$. In other words,
$$\delta(\mu) = P(- |\mu - \alpha|) - \Delta_{\alpha} \quad \mbox{if} \  \mu \in I_{\alpha}.$$
The function $\delta(\mu)$ is equal to $1/2$ at the end points of the interval $I_{\alpha}$.
The union of the intervals $\bigcup_{\alpha\in \cE} I_\alpha$ covers all rational numbers, but the complement $$C := \mathbb{R} - \bigcup_{\alpha \in \mathcal{E}} I_{\alpha}$$ is a Cantor set.  The graph of $\delta(\mu)$ intersects the line $\Delta = \frac{1}{2}$ precisely along $C$.  An important fact is that any point of the Cantor set is either an endpoint of $I_{\alpha}$ for some $\alpha$ or is a transcendental number \cite[Theorem 4.1]{CHW}. 

\subsection{Orthogonal parabolas and corresponding exceptional bundles} Here we recall the construction of the corresponding exceptional bundles to a Chern character $\bv$ with a positive dimensional moduli space $M(\bv)$.  

\subsubsection{Orthogonal parabolas}\label{sss-orthogonalParabolas} A Chern character $\bv$ of nonzero rank determines a $2$-plane $\bv^\perp \subset K(\P^2)\te \RR$ of characters $\bw$ with $\chi(\bv \te \bw) = 0$.  The projectivization of $\bv^\perp$ can be viewed as the \emph{orthogonal parabola to $\bv$} in the $(\mu,\Delta)$-plane, given explicitly by the Riemann--Roch formula by  $$\bv^\perp: \Delta = P(\mu+\mu(\bv)) - \Delta(\bv).$$ Orthogonal parabolas open upwards, and if $\bv$ has positive rank then $\chi(\bv \te \bw) > 0$ for characters $\bw$ \emph{below} the orthogonal parabola to $\bv$.  The orthogonal parabolas to two characters $\bv_1,\bv_2$ with $\mu(\bv_1)\neq \mu(\bv_2)$ intersect at a unique point, corresponding to the intersection $\bv_1^\perp \cap \bv_2^\perp$ in $K(\P^2) \te \RR$. Conversely, given two points $(\mu_1,\Delta_1)$ and $(\mu_2,\Delta_2)$ with $\mu_1\neq \mu_2$, there is, up to scale, a unique character $\bv$ such that the orthogonal parabola to $\bv$ contains them both.

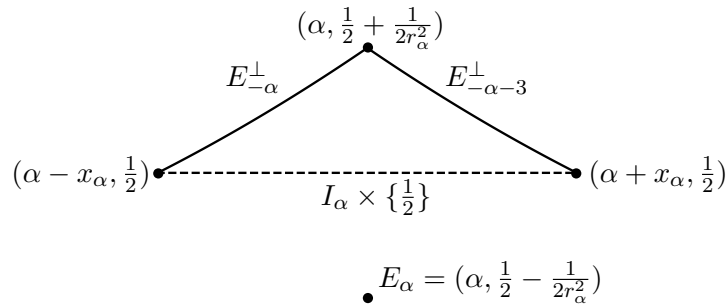
\begin{figure}[t] 
\begin{center}
\setlength{\unitlength}{1in}
\begin{picture}(2.29,1.61)
\put(0,0){\includegraphics[scale=.5,bb=0 0 4.58in 3.22in]{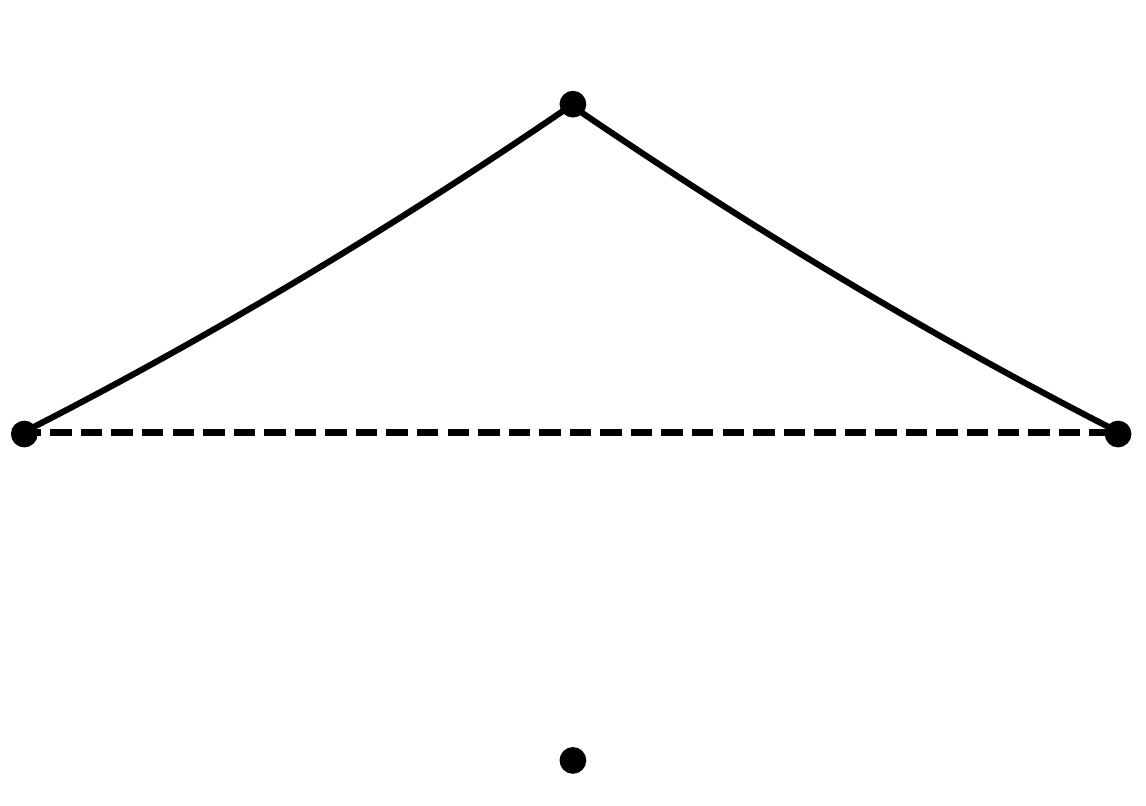}}
\put(.4,1.2){$E_{-\alpha}^\perp$}
\put(1.55,1.2){$E_{-\alpha-3}^\perp$}
\put(1.2,.15){$E_\alpha = (\alpha,\frac{1}{2}-\frac{1}{2r_\alpha^2})$}
\put(.9,.58){$I_\alpha \times \{\frac{1}{2}\}$}
\put(-.72,.7){$(\alpha-x_\alpha,\frac{1}{2})$}
\put(2.3,.7){$(\alpha+x_\alpha,\frac{1}{2})$}
\put(.77,1.48){$(\alpha,\frac{1}{2}+\frac{1}{2r_\alpha^2})$}
\end{picture}
\end{center}
\caption{The graph of the Dr\'ezet--Le Potier curve $\Delta = \delta(\mu)$ over the interval $I_\alpha$.}\label{fig-deltaInterval}
\end{figure}

\begin{example}
We graph the Dr\'ezet--Le Potier curve $\Delta = \delta(\mu)$ over the interval $I_\alpha$ and note its key features in Figure \ref{fig-deltaInterval}.  On the interval $(\alpha-x_\alpha,\alpha]$, the graph consists of the orthogonal parabola to $E_{-\alpha}$.  On the interval $[\alpha,\alpha+x_\alpha)$, the graph consists of the orthogonal parabola to $E_{-\alpha-3}$.
\end{example}

\subsubsection{Corresponding exceptional bundles} Let $\bv$ be a Chern character such that $M(\bv)$ is positive-dimensional.  By \cite[Theorem 3.1]{CHW}, the orthogonal parabola to $\bv$ intersects the line $\Delta=\frac{1}{2}$ in two points lying in segments $I_{\nu^-} \times \{\frac{1}{2}\}$ and $I_{\nu^+} \times \{\frac{1}{2}\}$ with $\nu^- < \nu^+$.  The bundles $E_{\nu^+}$ and $E_{\nu^-}$ are the \emph{primary and secondary corresponding exceptional bundles to $\bv$}, respectively.

\begin{remark}\label{rem-corrExcGap}
Since $\Delta(\bv) > \frac{1}{2}$, the orthogonal parabola to $\bv$ intersects the line $\Delta = \frac{1}{2}$ in two points that are at least $3$ units apart.  From this observation it easily follows that $\nu^+-\nu^- \geq 3$.
\end{remark}

\begin{remark}
The exceptional bundle $E_{\nu^-}$ can also be defined as the dual of the primary corresponding exceptional bundle to the Serre dual character $\bv^D$.  
\end{remark}

In the introduction a different definition of the primary corresponding exceptional bundle was given; the next result says that it is equivalent to the definition given above.  The proof is a straightforward application of the Intermediate Value Theorem and diagrams like Figure \ref{fig-deltaInterval}.

\begin{proposition}\label{prop-EulerExceptional}
Let $E_\beta$ be an exceptional bundle, and let $\bv$ be a Chern character such that $M(\bv)$ is positive dimensional.  Let $E_{\nu^{\pm}}$ be the corresponding exceptional bundles to $\bv$.
\begin{enumerate}
\item If $\beta < \nu^{-}$ or $\beta > \nu^+$, then $\chi(\bv \te E_\beta) >0$.
\item If $\nu^- < \beta < \nu^+$, then $\chi(\bv \te E_\beta) <0$.
\end{enumerate}
Thus $E_{\nu^{\pm}}$ are uniquely characterized as the exceptional bundles $E_\beta$ where the sign of $\chi(\bv\te E_\beta)$ changes.
\end{proposition}

We can also describe the characters $\bv$ with a given corresponding exceptional bundle.  Again the proof uses just the Intermediate Value Theorem.  See Figure \ref{fig-corrExc}.

\begin{proposition}\label{prop-corrExc}
Let $E_\beta$ be an exceptional bundle, and let $\bv$ be a non-exceptional stable Chern character.  
\begin{enumerate}
\item The bundle $E_\beta$ is the primary corresponding exceptional bundle to $\bv$ if and only if $$\chi(\bv \te (1,\beta-x_\beta,1/2)) < 0 \quad \textrm{and}\quad \chi(\bv \te (1,\beta+x_\beta,1/2)) >0.$$
\item The bundle $E_\beta$ is the secondary corresponding exceptional bundle to $\bv$ if and only if $$\chi(\bv \te (1,\beta-x_\beta,1/2)) > 0 \quad \textrm{and}\quad \chi(\bv \te (1,\beta+x_\beta,1/2)) < 0.$$
\end{enumerate}
\end{proposition}

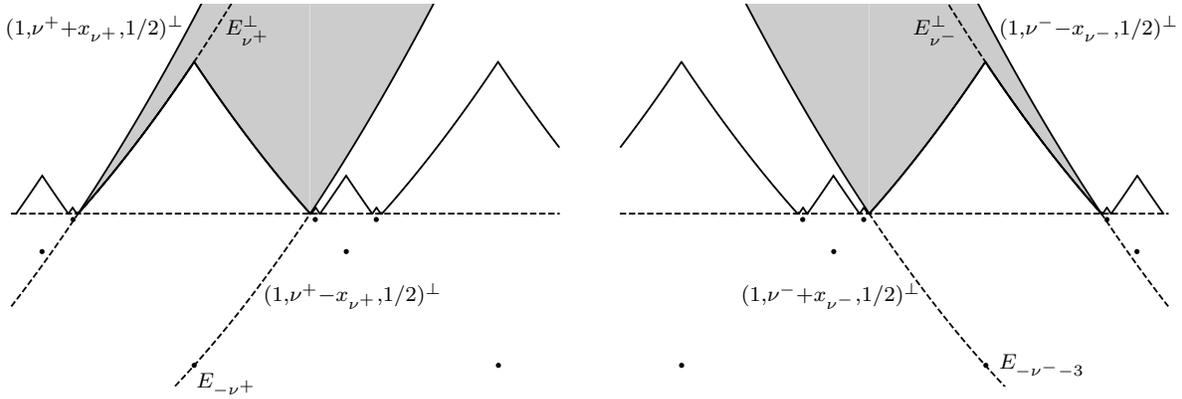
\begin{figure}[t] 
\begin{center}

\setlength{\unitlength}{1in}
\begin{picture}(6.5,2.3)
\put(0,0){\includegraphics[scale=.39,bb=0 0 16.49in 5.9in]{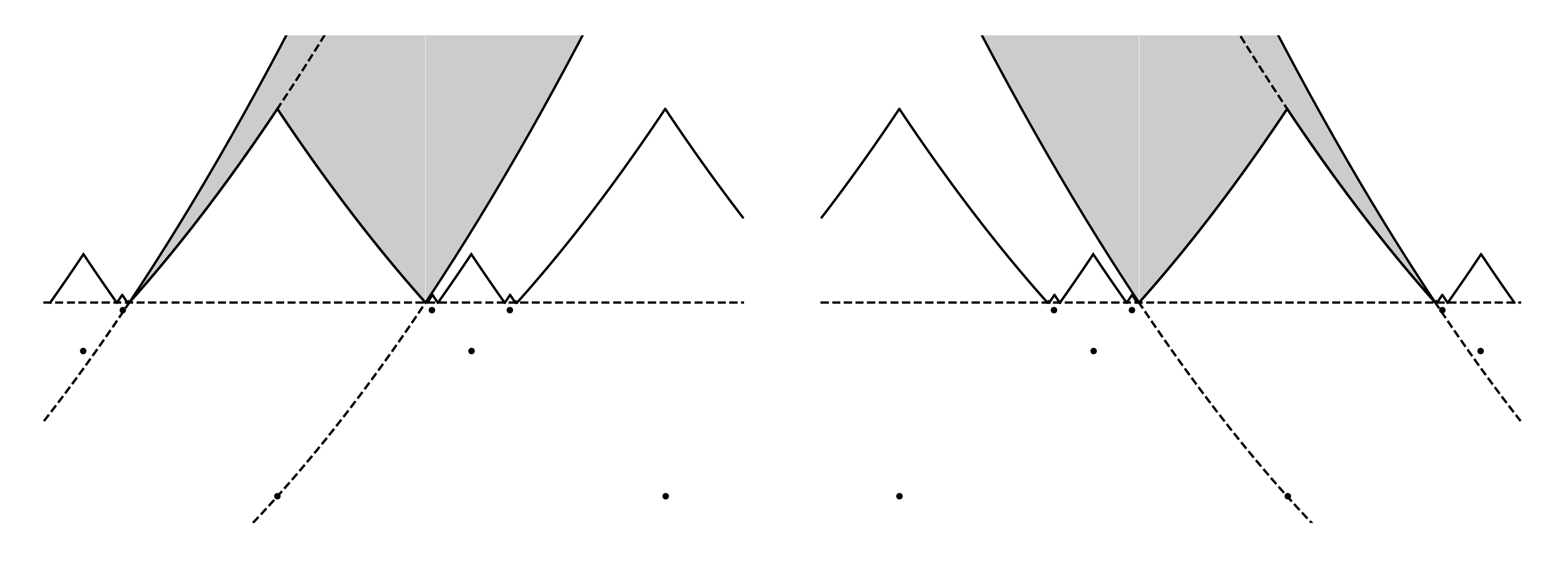}}
\put(1.3,2){$\scriptstyle E_{\nu^+}^\perp$}
\put(1.15,.15){$\scriptstyle E_{-\nu^+}$}
\put(1.5,.6){${\scriptstyle(1,\nu^+-x_{\nu^+},1/2)^\perp}$}
\put(.15,2){${\scriptstyle(1,\nu^++x_{\nu^+},1/2)^\perp}$}
\put(4.9,2){$\scriptstyle E_{\nu^-}^\perp$}
\put(5.35,.25){$\scriptstyle E_{-\nu^--3}$}

\put(5.35,2){${\scriptstyle(1,\nu^--x_{\nu^-},1/2)^\perp}$}
\put(4,.6){${\scriptstyle(1,\nu^-+x_{\nu^-},1/2)^\perp}$}
\end{picture}
\end{center}
\caption{In the diagram on the left, the shaded region indicates the stable characters $\bv$ with primary corresponding exceptional bundle $E_{\nu^+}$.  Characters below the dotted parabola given by $E_{\nu^+}^\perp$ have $\chi(\bv \te E_{\nu^+}) > 0$, and the opposite inequality holds on the other side.  On the right, the shaded region indicates stable characters with secondary corresponding exceptional bundle $E_{\nu^-}$.  See Proposition \ref{prop-corrExc}.}\label{fig-corrExc}
\end{figure}

\subsection{Beilinson spectral sequences and the Kronecker fibration} Here we discuss how to use exceptional collections as building blocks for arbitrary sheaves.  We then discuss the Kronecker fibration $M(\bv) \dashrightarrow Kr(\bv)$ which associates a two-term complex $K$ to a general sheaf $V\in M(\bv)$.

\subsubsection{Beilinson spectral sequences}
Following \cite{DrezetBeilinson}, we define a \emph{triad} of exceptional bundles on $\P^2$ to be a collection $(E,G,F)$ of exceptional bundles whose slopes are of one of the form $$(\beta-3,\alpha,\alpha.\beta), \quad (\alpha,\alpha.\beta,\beta),\quad \textrm{or}\quad(\alpha.\beta,\beta,\alpha+3)$$ for some exceptional slopes $\alpha,\beta$ with $\alpha = \varepsilon(p/2^q)$ and $\beta = \varepsilon((p+1)/2^q)$.  Any triad is a full strong exceptional collection for the derived category $D^b(\P^2)$.   Corresponding to the triad $(E,G,F)$ is a fourth exceptional bundle $M$ defined as the cokernel of a canonical coevaluation mapping $$0\to G\to F\te \Hom(G,F)^* \to M\to 0.$$ The collection $(E^*(-3),M^*,F^*)$ is again a triad called the \emph{dual triad} to $(E,G,F)$.  The Beilinson spectral sequence allows us to decompose any sheaf $V$ on $\P^2$ in terms of the triad $(E,G,F)$.

\begin{theorem}[\cite{DrezetBeilinson}]
Let $V$ be a coherent sheaf on $\P^2$, and let $(E,G,F)$ be a triad.  Write
$$\begin{aligned}
G_{-2} &= E \\
G_{-1} &= G \\
G_0 &= F 
\end{aligned}\qquad\qquad
\begin{aligned} 
F_{-2} &= E^*(-3)\\
F_{-1} &= M^*\\
F_{0} &= F^*,
\end{aligned}$$ and put $G_i=F_i=0$ if $i\notin \{-2,-1,0\}$.
There is a spectral sequence with $E_1^{p,q}$-page $$E_{1}^{p,q} = G_p\te H^q(V\te F_p)$$ converging to $V$ in degree $0$ and to $0$ in all other degrees.
\end{theorem}

\subsubsection{Moduli of two-term complexes}\label{sssec-complexes} Let $(E,F)$ be an exceptional pair, i.e., a pair of exceptional bundles with $\Ext^i(E,F)=0$ for $i>0$ and $\Ext^i(F,E) =0$ for all $i$.  Consider two complexes of the form $K:E^{b}\to F^{a}$ and $K':E^{b'}\to F^{a'}$, each supported in degrees $-1$ and $0$.  A spectral sequence calculation (see \cite[Lemma 5.4]{CHW}) shows that the homomorphisms $K'\to K$ in $D^b(\P^2)$ are exactly the homomorphisms of complexes, and so are given by commuting diagrams  
$$\xymatrix{
E^{b'}\ar[d] \ar[r] & F^{a'}  \ar[d]\\
E^b \ar[r] & F^a
}$$
where the vertical maps are scalar matrices.  In particular, two complexes of the form  $E^b\to F^a$ are isomorphic if and only if they are conjugate under the natural action of $\GL(b)\times \GL(a)$.  By associating a complex $E^b \to F^a$ to the linear map $\CC^b\to \CC^a \te \Hom(E,F)^*$ we see that the full subcategory of $D^b(\P^2)$ consisting of complexes of this form is an abelian category equivalent to the category of Kronecker $\Hom(E,F)^*$-modules (see \cite[Section 6]{CHW} or \S\ref{sec-kronecker} for definitions). In particular, we can compute $\Ext^i(K',K)$ by first passing to the associated Kronecker modules.

\subsubsection{The Kronecker fibration when $\chi(V\te E_{\nu^+})>0$}\label{sss-res1} Following \cite[Section 5]{CHW}, let $\bv$ be a character such that $M(\bv)$ is positive dimensional, and let $E_{\nu^+}$ be the primary corresponding exceptional bundle.  Assume $\chi(V\te E_{\nu^+})>0$.  We write $\nu^+ = \alpha.\beta$ for some exceptional slopes $\alpha$ and $\beta$, and we consider the triad $(E_{-\alpha-3},E_{-\beta},E_{-\nu^+})$ with dual triad $(E_\alpha,E_{\alpha.\nu^+},E_{\nu^+})$. The Beilinson spectral sequence of the general $V\in M(\bv)$ takes the form
 $$\xymatrix{
E_{-\alpha-3}^{m_1} \ar[r] & E_{-\beta}^{m_2}  \ar[r] & 0  \\
0\ar[r] & 0  \ar[r] & E_{-\nu^+}^{m_3}  \\
}$$
where
$$
m_1 =-\chi(V\te E_\alpha) \qquad m_2 = -\chi(V\te E_{\alpha.\nu^+})\qquad
m_3 = \chi(V\te E_{\nu^+}) = \hom(E_{-\nu^+},V).
$$
Let $K\in D^b(\P^2)$ be the two-term complex $$K:E_{-\alpha-3}^{m_1}\to E_{-\beta}^{m_2}$$ in degrees $-1$ and $0$.  Then $K$ is isomorphic to the mapping cone of the canonical evaluation $E_{-\nu^+}^{m_3}\to V$, so there is a triangle $$E_{-\nu^+}^{m_3}\to V\to K\to \cdot.$$

The complex $K$ can also be arrived at in another way.  If we consider the family of cokernels $$0\to E_{-\alpha-3}^{m_1}\fto{\begin{pmatrix}\scriptstyle\phi_1\\ \scriptstyle\phi_2\end{pmatrix}} E_{-\beta}^{m_2}\oplus E_{-\nu^+}^{m_3} \to V \to 0$$ parameterized by $S = \Hom(E_{-\alpha-3}^{m_1},E_{-\beta}^{m_2}\oplus E_{-\nu^+}^{m_3})$, then the general cokernel $V$ is a general stable sheaf and $K$ is isomorphic to the two-term complex given by $\phi_1$.  Consequently, as $V$ varies in moduli, the map defining the complex $K$ is also general. In \cite[Proposition 6.3]{CHW} it is shown that this implies that the corresponding Kronecker module is stable (see \S\ref{sec-kronecker} for stability of Kronecker modules).  Then if $Kr(\bv)$ is the corresponding moduli space of semistable Kronecker modules, the Kronecker fibration is the dominant rational map $M(\bv)\dashrightarrow Kr(\bv)$ sending $V\in M(\bv)$ to the isomorphism class of the Kronecker module associated to $K$.  The general fiber of the map is positive-dimensional and birational to a Grassmannian.

\begin{remark}
Since $K$ is given by a general map $\phi_1$ and $\sHom(E_{-\alpha-3},E_{-\beta})$ is globally generated (see Theorem \ref{thm-excgg}), we see by a Bertini-type theorem \cite[Proposition 2.6]{HuizengaJAG} that unless $\rk(K) = -1$, $K$ is isomorphic to either a sheaf in degree $-1$ or a sheaf in degree $0$.  If $\rk(K) = -1$, then typically $K$ is not isomorphic to a shift of a sheaf.  To handle things in a uniform fashion it is best to treat $K$ as a complex.
\end{remark}

\subsubsection{The Kronecker fibration when $\chi(V\te E_{\nu^+})\leq 0$}\label{sss-res2}

A similar discussion holds here, so we omit some details.  When $\chi(V\te E_{\nu^+})\leq 0$, we again write $\nu^+ = \alpha.\beta$, but this time we use the triad $(E_{-\nu^+-3},E_{-\alpha-3},E_{-\beta})$ with dual triad $(E_{\nu^+},E_{\nu^+.\beta},E_{\beta})$.  The spectral sequence of a general $V\in M(\bv)$ then looks like 
 $$\xymatrix{
E_{-\nu^+-3}^{m_3} \ar[r] & 0 \ar[r] & 0  \\
0\ar[r] & E_{-\alpha-3}^{m_1}  \ar[r] & E_{-\beta}^{m_2}  \\
}$$
where
$$
m_1 =\chi(V\te E_{\nu^+.\beta}) \qquad m_2 = \chi(V\te E_{\beta})\qquad
m_3 = -\chi(V\te E_{\nu^+}) = \ext^1(E_{-\nu^+-3},V).
$$
Since the spectral sequence converges to a sheaf in degree $0$, we get short exact sequences of sheaves $$0\to E_{-\nu^+-3}^{m_3}\to K\to V\to 0$$ $$0\to E_{-\alpha-3}^{m_1}\to E_{-\beta}^{m_2} \to K\to 0$$ (in particular the two-term complex $E_{-\alpha-3}^{m_1}\to E_{-\beta}^{m_2}$ is isomorphic to a sheaf $K$).  As before, the general $V$ can also be constructed by a resolution $$0\to E_{-\nu^+-3}^{m_3}\oplus E_{-\alpha-3}^{m_1}\fto{\scriptstyle (\phi_1,\phi_2)}E_{-\beta}^{m_2}\to V\to 0,$$ and $K$ is isomorphic to the complex given by $\phi_2$.  Therefore $K$ is general and is actually a stable sheaf.  The corresponding Kronecker module is also stable.  The Kronecker fibration $M(\bv)\dashrightarrow Kr(\bv)$ sends $V$ to the isomorphism class of the Kronecker module corresponding to $K$.  The general fiber is positive dimensional if $\chi(V\te E_{\nu^+})<0$.  When $\chi(V\te E_{\nu^+})=0$, the map contracts the Brill--Noether divisor consisting of those sheaves with $h^0(V\te E_{\nu^+})>0$.

\subsection{Prioritary sheaves}\label{ss-prioritary} A torsion-free coherent sheaf $E$ on $\PP^2$ is called {\em prioritary} if $$\Ext^2(E, E(-1))=0.$$ Compared to semistable sheaves, prioritary sheaves are easier to construct.  Hirschowitz and Laszlo \cite{HirschowitzLaszlo} prove that the stack of prioritary sheaves on $\PP^2$ with Chern character $\bv$ is an irreducible stack whenever nonempty. This stack contains the stack of semistable sheaves as a (possibly empty) open substack. Hence, to prove a general semistable sheaf has an open property, it suffices to construct a prioritary sheaf with that property. By semicontinuity, to show the vanishing of a cohomology group for the general semistable sheaf, it suffices to produce a prioritary sheaf with the required vanishing.

\section{Homomorphisms of Kronecker modules}\label{sec-kronecker}

The Kronecker fibration of $M(\bv)$ allows us to reduce many interesting computations of the cohomology of a tensor product $V\te W$ to computations of homomorphisms between two-term complexes given by exceptional pairs.  By the discussion in \S\ref{sssec-complexes}, these spaces of homomorphisms can be computed by instead studying homomorphisms between Kronecker modules, which we now do in more detail.

\subsection{Kronecker modules} Following \cite{Drezet}, let $N\geq 3$ and let $V$ be a vector space of dimension $N$.  A \emph{Kronecker $V$-module} is a linear map $$e: \CC^b\to \CC^a\te V^*,$$ or a matrix $e\in \Mat_{a\times b}(V^*)$ with entries in $V^*$. The \emph{dimension vector} of $e$ is $\udim(e) = (b,a)$.  If $f:\CC^{b'}\to \CC^{a'}\te V^*$ is another Kronecker $V$-module, then a homomorphism from $f$ to $e$ is a pair of matrices $\beta \in \Mat_{b\times b'}(\CC)$ and $\alpha\in \Mat_{a\times a'}(\CC)$ such that the diagram 
$$\xymatrix{
\CC^{b'}\ar[d]^{\beta} \ar[r]^f & \CC^{a'} \te V^* \ar[d]^{\alpha\te \id}\\
\CC^b \ar[r]^e & \CC^a \te V^*
}$$
commutes.  The category of Kronecker $V$-modules is an abelian category.  We have $\Ext^i(f,e) = 0$ for $i>1$, and  the \emph{Euler characteristic} $$\chi (f,e) := \hom(f,e) - \ext^1(f,e) = b'b+a'a-Nb'a $$ depends only on the dimension vectors of $f$ and $e$.  A Kronecker module can also be viewed as a representation of the $N$-arrowed Kronecker quiver.

\subsection{Semistability} The \emph{slope} of a nonzero Kronecker module $e$ of dimension vector $(b,a)$ is defined to be $\mu(e) = b/a$, interpreted as $\infty$ if $a=0$ and $b>0$.  The module $e$ is \emph{(semi)stable} if every submodule $f\subset e$ has $\mu(f) \leqor \mu(e)$.  There is a moduli space $Kr_N(b,a)$ parameterizing $S$-equivalence classes of semistable Kronecker modules of dimension vector $(b,a)$.  It can be constructed as a GIT quotient of $\Hom(\CC^b,\CC^a\te V^*)$ by the group $(\SL(b)\times \SL(a))/\CC^*$.  The \emph{expected dimension} of the moduli space is $$\edim(Kr_N(b,a))=1-\chi(e,e)=1-b^2-a^2+Nba.$$  If there is a stable module of dimension vector $(b,a)$, then the expected dimension of $Kr_N(b,a)$ is nonnegative. 
\begin{theorem}[\cite{Drezet}]
If $Kr_N(b,a)$ has nonnegative expected dimension, then it is nonempty and irreducible of the expected dimension, and the general module is stable.
\end{theorem}

Observe that the expected dimension is at least $1$ if $$b^2+a^2-Nba \leq 0.$$  Rewriting this in terms of the slope $\mu = b/a$, the expected dimension is at least $1$ if and only if $$\mu^2-N\mu+1\leq 0.$$  The two roots of $\mu^2-N\mu +1=0$ are $\psi_N^{-1}$ and $\psi_N$, where $$\psi_N = \frac{N+ \sqrt{N^2-4}}{2},$$ so the moduli space  $Kr_N(b,a)$ is positive dimensional if and only if $$\mu = \frac{b}{a} \in (\psi_N^{-1},\psi_N).$$ (Note that $\psi_N$ is irrational since $N\geq 3$.)  If $\mu$ is outside of this interval and there is a stable module, then the expected dimension must be $0.$   In this case the moduli space $Kr_N(b,a)$ is a single reduced point, corresponding to an \emph{exceptional Kronecker module}.  The dimension vectors for these modules are easy to describe.  We must have $$b^2+a^2-Nba = 1.$$  Up to swapping $a$ and $b$, the nonnegative solutions of this Diophantine equation can be obtained from the solution $(0,1)$ by repeatedly applying the transformation $\tau(b,a) = (a,Na-b)$.  

\begin{example}
For $N=3$, the orbit of $(0,1)$ under powers of $\tau$ is as follows:
$$(0,1) \mapsto ( 1,3) \mapsto (3,8) \mapsto (8,21) \mapsto (21,55) \mapsto \cdots.$$ The corresponding slopes $$\frac{0}{1},\frac{1}{3},\frac{3}{8},\frac{8}{21},\frac{21}{55},\ldots$$ are increasing and converge to $\psi_3^{-1}$.  There is a stable module of slope $\mu$ if and only if $$\mu \in \left\{\frac{0}{1},\frac{1}{3},\frac{3}{8},\frac{8}{21},\frac{21}{55},\ldots\right\} \cup (\psi_3^{-1},\psi_3) \cup \left\{\ldots,\frac{55}{21},\frac{21}{8},\frac{8}{3},\frac{3}{1},\infty\right\}.$$ A similar result holds for arbitrary $N\geq 3$.
\end{example}

Suppose $(b,a)$ is a dimension vector such that $\mu=b/a$ is not in the interval $(\psi_N^{-1},\psi_N)$.  Then as in \cite{Schofield} the general Kronecker module with dimension vector $(b,a)$ can be described as follows.  There are two exceptional Kronecker modules $s_1$ and $s_2$ such that $\mu(s_1)\leq \mu \leq \mu(s_2)$ and such that no exceptional module has slope strictly between $\mu(s_1)$ and $\mu(s_2)$.  Then the module given by a general map $\CC^b\to \CC^a \te V^*$ is isomorphic to a direct sum $s_1^{\oplus m_1}\oplus s^{\oplus m_2}$, with the exponents being easily determined from the dimension vector.

\subsection{Homomorphisms between general  Kronecker modules}

We now study the following problem.  Suppose $f,e$ are general Kronecker modules.  Then we can hope that at most one of $\Hom(f,e)$ or $\Ext^1(f,e)$ is nonzero, and so $\chi(f,e)$ determines both spaces.  The next example shows this is too optimistic.

\begin{example}
Consider $N=3$ and let $f$ and $e$ be general with $\udim f = (1,4)$ and $\udim e = (4,11)$.  Then neither $f$ or $e$ are semistable.  Let $s'',s',s$ be the exceptional modules of vectors $\udim s'' = (0,1)$, $\udim s'=(1,3)$, and $\udim s=(3,8)$.  We have $f = s''\oplus s'$ and $e= s'\oplus s$, so $\Hom(f,e)\neq 0$.  However, $\chi(f,e)=0$, so also $\Ext^1(f,e) \neq 0$.
\end{example}

On the other hand, if one of the modules is semistable, then the Euler characteristic governs $\Hom(f,e)$.

\begin{theorem}\label{thm-kronecker}
Let $f$ and $e$ be Kronecker modules which are general of their dimension vectors.  Suppose that one of the modules is semistable.  Then at most one of the groups $\Ext^i(f,e)$ is nonzero.
\end{theorem}

\begin{proof}
Let $f, e$ have dimension vectors $(b',a'),(b,a)$ and slopes $\mu',\mu$.  The result is easy if $\mu'=0$ or $\mu=\infty$, so assume we are not in these cases.  We assume $e$ is semistable; a symmetric argument handles the other possibility.

We perform some reductions to assume that $f$ and $e$ are both stable.    By a straightforward argument with Jordan--H\"older filtrations, we may as well assume $e$ is stable. If $f$ is not stable, then it is a direct sum $f = s_1^{\oplus m_1}\oplus s_2^{\oplus m_2}$ of copies of two ``adjacent'' exceptional modules $s_1$ and $s_2$.  We claim that the numbers $\chi(f,e)$, $\chi(s_1,e)$, $\chi(s_2,e)$ are all either nonpositive or nonnegative. For any module $g$, the sign of the Euler characteristic $\chi(g,e)$ depends only on the slope of $g$, and we have $\chi(g,e) = 0$ if and only if $\mu(g) = 1/(N-\mu)$.  If $\mu\in (\psi_N^{-1},\psi_N)$ then we have $1/(N-\mu)\in (\psi_N^{-1},\psi_N).$  On the other hand, if $(b,a)$ is the vector of an exceptional module, then so is $(a,Na-b)$, and its slope is $1/(N-\mu)$.  Therefore in either case $1/(N-\mu)$ does not lie strictly between $\mu(s_1)$ and $\mu(s_2)$, and the three numbers $\chi(f,e),\chi(s_1,e),\chi(s_2,e)$ must all be either nonpositive or nonnegative.  Thus if $f$ is not stable, the result follows from the result for $s_1$ and $s_2$.

For the rest of the proof we assume $f$ and $e$ are both stable.  If $\mu(f) \geq \mu(e)$ then the conclusion follows from stability, so we assume $\mu(f)< \mu(e)$.  View $f$ and $e$ as $a'\times b'$ and $a\times b$ matrices with entries in $V^*$, respectively.  We consider the incidence correspondence
\begin{align*} \Sigma& \subset \P(\Mat_{b\times b'}(\CC)\times \Mat_{a\times a'}(\CC))\times (\Mat_{a'\times b'}(V^*)\times\Mat_{a\times b}(V^*))\\ \Sigma &= \left\{([\beta:\alpha],(f,e)): {e\beta = \alpha f\atop \textrm{$e,f$ stable}}\right\} \end{align*} with projections 
\begin{align*}\pi_1:\Sigma &\to \P(\Mat_{b\times b'}(\CC)\times \Mat_{a\times a'}(\CC))\\
\pi_2:\Sigma &\to \Mat_{a'\times b'}(V^*) \times \Mat_{a\times b}(V^*)
\end{align*}
We will show that $$\dim \Sigma \leq \max\{0,\chi(f,e)\}+\dim (\Mat_{a'\times b'}(V^*)\times \Mat_{a\times b}(V^*)) -1.$$ Then if $\chi(f,e) \leq 0$, this implies $\pi_2$ is not dominant and so $\Hom(f,e) = 0$ for general $f$ and $e$.  On the other hand if $\chi(f,e) > 0$, then we see that the general fiber of $\pi_2$ has dimension at most $\chi(f,e)-1$.  Therefore $\hom(f,e)  \leq \chi(f,e)$, but $\hom(f,e) \geq \chi(f,e)$ always holds and so $\hom(f,e) = \chi(f,e)$.

We study the dimension of $\Sigma$ by analyzing the first projection.  However, the fibers of the first projection jump over special pairs of matrices.  In particular, as the ranks of $\beta$ and $\alpha$ drop, there are more pairs $(f,e)$ satisfying $e\beta = \alpha f$; consider $\beta=\alpha=0$ for an extreme example.  Thus for nonnegative integers $r$ and $s$ (not both zero) we let $$\Sigma_{r,s}=\{([\beta:\alpha],(f,e))\in \Sigma:\rk\beta = r \textrm{ and } \rk\alpha = s\}.$$ Then $\Sigma$ is covered by the various $\Sigma_{r,s}$, so we turn to estimating the dimension of $\Sigma_{r,s}$.  

First observe that if $\Sigma_{r,s}$ is nonempty, then a point $([\beta:\alpha],(f,e))\in\Sigma_{r,s}$ gives a map $f\to e$ of stable modules whose image has dimension vector $(r,s)$.    By stability, this forces $(r,s)$ to satisfy $$\frac{b'}{a'}\leq \frac{r}{s}\leq \frac{b}{a}.$$ So, in what follows we assume $(r,s)$ satisfies these inequalities.

Suppose $\beta$ and $\alpha$ have rank $r$ and $s$, respectively.  Then the fiber $\pi_1^{-1}([\beta:\alpha])$ is identified with those pairs $(f,e)$ such that $$\xymatrix{
\CC^{b'}\ar[d]^{\beta} \ar[r]^f & \CC^{a'} \te V^* \ar[d]^{\alpha\te \id}\\
\CC^b \ar[r]^e & \CC^a \te V^*
}$$ commutes.  But if we change bases on the spaces $\CC^{b'},\CC^{b},\CC^{a'},\CC^{a}$, then we may assume $\beta$ and $\alpha$ are of the form $$\beta = \begin{pmatrix}I_r & 0 \\ 0 & 0 \end{pmatrix} \qquad \textrm{and} \qquad\alpha = \begin{pmatrix} I_s & 0 \\ 0 & 0\end{pmatrix}.$$ Writing $f = (f_{ij})$ and $e= (e_{ij})$ with\ $f_{ij},e_{ij} \in V^*$, we compute $$e\beta = \begin{pmatrix} e_{11}& \cdots & e_{1r} & 0 &\cdots & 0 \\ \vdots & \ddots & \vdots &\vdots & \ddots & \vdots\\ e_{s1} & \cdots & e_{sr} & 0 & \cdots  & 0\\
e_{s+1,1}& \cdots & e_{s+1,r} & 0 &\cdots & 0 \\ \vdots & \ddots & \vdots &\vdots & \ddots & \vdots \\ e_{a1} & \cdots & e_{ar} & 0 & \cdots  & 0
\end{pmatrix}$$and $$\alpha f = \begin{pmatrix} f_{11}& \cdots & f_{1r} & f_{1,r+1} & \cdots & f_{1,b'}\\ \vdots & \ddots & \vdots&\vdots & \ddots & \vdots \\ f_{s1} & \cdots & f_{sr} & f_{s,r+1} & \cdots & f_{s,b'}\\ 0 & \cdots & 0 & 0 & \cdots & 0\\ \vdots & \ddots & \vdots& \vdots & \ddots & \vdots\\  0 & \cdots & 0 & 0 & \cdots & 0
\end{pmatrix}.
$$Thus the condition that $e\beta = \alpha f$ requires that
\begin{itemize} \item $e_{ij} = f_{ij}$ for $1\leq i \leq s$ and $1\leq j \leq r$, 
\item $e_{ij} = 0$ for $s+1\leq i \leq a$ and $1\leq j \leq r$, and 
\item $f_{ij} = 0$ for $1\leq i \leq s$ and $r+1\leq j\leq b'.$\end{itemize}  Each of these conditions imposes $N$ independent linear conditions on  $\Mat_{a'\times b'}(V^*)\times \Mat_{a\times b}(V^*).$ Since we further require $f$ and $e$ to be stable, we conclude $$\dim \pi_1^{-1}([\beta:\alpha]) \leq N(a'b'+ab-ab'+(a-s)(b'-r)).$$ Letting $U_r\subset \Mat_{b\times b'}(\CC)$ and $V_s\subset \Mat_{a\times a'}(\CC)$ be the subsets of matrices of rank $r$ and $s$, respectively, we have \begin{align*}\dim U_r &= bb' - (b-r)(b'-r)\\ \dim V_s &= aa' - (a-s)(a'-s).\end{align*} Then $\Sigma_{r,s} = \pi_1^{-1}(\P(U_r\times V_s))$, so \begin{align*}\dim \Sigma_{r,s} &\leq aa'+bb'+N(a'b'+ab-ab'+(a-s)(b'-r))\\&\qquad-(b-r)(b'-r)-(a-s)(a'-s)-1\\
&= \chi(f,e)+\dim (\Mat_{a'\times b'}(V^*)\times \Mat_{a\times b}(V^*))-1\\&\qquad +N(a-s)(b'-r)-(b-r)(b'-r)-(a-s)(a'-s).\end{align*} 

We now define and study an auxiliary function $$Q(s,r) := N(a-s)(b'-r)-(b-r)(b'-r)-(a-s)(a'-s)$$ on the domain $$\Omega = \left\{(s,r):0\leq s\leq \min\{a,a'\},0\leq r \leq \min\{b,b'\},\frac{b'}{a'}s\leq r \leq \frac{b}{a}s\right\}\subset \RR^2,$$  with the goal of showing $Q(s,r) \leq  \max\{0,-\chi(f,e)\}.$  Observe that if $\Sigma_{r,s}$ is nonempty then $(s,r)\in \Omega$, so from this it will follow that $$\dim \Sigma_{r,s} \leq \max\{0,\chi(f,e)\}+\dim(\Mat_{a'\times b'}(V^*)\times \Mat_{a\times b}(V^*))-1,$$ as required.  The function $Q(s,r)$ is a quadratic function of $r$ and $s$.  Its Hessian determinant $$\begin{vmatrix} Q_{ss} & Q_{sr} \\ Q_{rs} & Q_{rr}\end{vmatrix} = \begin{vmatrix}-2 & N \\ N & -2\end{vmatrix} = 4 -N^2$$ is negative since $N\geq 3$, so $Q$ has no local maximum and its maximum value on $\Omega$ is attained on the boundary.  The boundary of $\Omega$ is made up of up to four line segments.

\emph{Case 1:} $r=\min\{b,b'\}$.  If $r= b$ then $r\leq \frac{b}{a}s$ gives $s\geq a$.  Since $a\leq \min\{a,a'\}$ we have $s=a$, but $Q(a,b)=0$.  If instead $r=b'$, then $$Q(s,b') = -(a-s)(a'-s) \leq 0$$ since $s\leq \min\{a,a'\}$.

\emph{Case 2:} $s=\min\{a,a'\}$.  If $s=a'$ then $r\geq \frac{b'}{a'}s$ gives $r \geq b'$ and so $r=b'$ since $r\leq \min\{b,b'\}$.  But $Q(a',b')=0$.  If we have $s=a$, then $$Q(a,r) = -(b-r)(b'-r)\leq 0$$ since $r\leq \min\{b,b'\}$.

\emph{Case 3:} $r= \frac{b}{a}s$.  In this case we compute\begin{align*}
a^2Q\left(s,\frac{b}{a}s\right)&=Na(a-s)(b'a-bs)-b(a-s)(b'a-bs)-a^2(a-s)(a'-s)\\
&=(a-s)(Na^2b'-Nabs-abb'+b^2s-a^2a'+a^2s)\\
&= (a-s)(-a\chi(f,e)+s\chi(e,e)).
\end{align*}
Stability gives $\chi(e,e) \leq 1$.  If $\chi(f,e)>0$ then since $s\leq a$ we get $Q(s,\frac{b}{a}s)\leq 0$.  If $\chi(e,e) = 1$, then by a straightforward computation the assumption $\mu(f) < \mu(e)$ gives $\chi(f,e) > 0$, so we are in the previous situation.   Finally, if $\chi(f,e)\leq 0$ and $\chi(e,e)\leq 0$ then $Q(s,\frac{b}{a}s)\leq -\chi(f,e)$.  Therefore $Q(\frac{b}{a}s,s)\leq \max\{0,-\chi(f,e)\}$ in every case.

\emph{Case 4:} $r= \frac{b'}{a'}s$.  This case follows from a formula $$(a')^2Q\left(s,\frac{b'}{a'}s\right) = (a'-s)(-a'\chi(f,e)+s\chi(f,f))$$ and similar reasoning to Case 3.
\end{proof}

\begin{remark}
In the theorem if neither $f$ nor $e$ is semistable then it is still straightforward to compute $\Hom(f,e)$.  In this case there are exceptional modules $s_1,s_2,s_3,s_4$ and decompositions $f = s_1^{\oplus m_1}\oplus s_2^{\oplus m_2}$ and $e= s_3^{m_3}\oplus s_4^{m_4}$.  The terms $\Hom(s_i,s_j)$ can all be computed using the theorem, and so $\Hom(f,e)$ can also be computed.
\end{remark}

\section{Twists by exceptional bundles}\label{sec-exceptional}

In this section we study the cohomology of a tensor product $V\te E$, where $V$ is a general stable bundle and $E$ is an exceptional bundle.  Our goal is to show that the cohomology of $V\te E$ is entirely determined by the Euler characteristic and the slope.  As an application, we show that if $V$ is a general stable bundle and $(E,F,G)$ is any triad of exceptional bundles, then the shape of the corresponding Beilinson spectral sequence for $V$ can be determined.  Along the way, we develop some basic facts about exceptional bundles as well as a criterion for a general vector bundle $\sHom(W,V)$ to be globally generated.

\subsection{Inductive description of exceptional bundles}\label{ss-inductExceptional}  We begin by recalling how to build up exceptional bundles in terms of simpler exceptional bundles.  Let $\beta = \varepsilon((p+1)/2^q)$ be an exceptional slope with $p$ an even integer and $q\geq 1$ an integer.  We consider the exceptional slopes
$$\zeta_0 = \varepsilon\left(\frac{p+4}{2^q}-3\right)\qquad \zeta_2 = \varepsilon\left(\frac{p-2}{2^q}\right) $$$$\alpha = \varepsilon\left(\frac{p\vphantom1}{2^q}\right) \qquad \beta = \varepsilon\left(\frac{p+1}{2^q}\right) \qquad \eta = \varepsilon\left(\frac{p+2}{2^q}\right)$$$$ \quad \omega_0=\varepsilon\left(\frac{p+4}{2^q}\right)\qquad \omega_2=\varepsilon\left(\frac{p-2}{2^q}+3\right),$$ where $p$ is even and $q\geq 1$.  Let $i\in\{0,2\}$ be such that $i\equiv p \pmod 4$, and observe that the exceptional slopes $\alpha,\eta,\zeta_i,\omega_i$ all have smaller order than $\beta$.  Dr\'ezet gives the following exact sequences which express $E_\beta$ in terms of exceptional bundles of smaller order.

\begin{theorem}[\cite{DrezetBeilinson}]\label{kernelSlope}
Let $i\in \{0,2\}$ be such that $i\equiv p \pmod 4$.  There are exact sequences of vector bundles$$0\to E_{\zeta_i}\to E_\alpha\te \Hom(E_\alpha,E_\beta)\to E_\beta\to 0$$ and $$0\to E_\beta\to E_{\eta}\te \Hom(E_{\beta},E_\eta)^*\to E_{\omega_i}\to 0.$$
\end{theorem}

As a consequence, the cohomology of a tensor product of exceptional bundles can be determined.  This theorem was first proved by Dr\'ezet, but we include the proof since we will be generalizing this line of reasoning in Section \S\ref{ss-excTwist}.

\begin{theorem}[\cite{DrezetBeilinson}]\label{thm-exceptionalHom}
Let $E$ and $F$ be exceptional bundles with $\mu(E)\leq \mu(F)$.  Then $\Ext^i(E,F) = 0$ for $i>0$.
\end{theorem}
\begin{proof}
We induct on the orders of $E$ and $F$.  The result is clear if $E$ and $F$ are both line bundles.   The result is also clear if $\mu(E)=\mu(F)$, so suppose $\mu(E)<\mu(F)$.  Then either $\ord(E) \leq \ord(F)$ or $\ord(E)\geq \ord(F)$; we handle each case individually.

\emph{Case 1:} $\ord(E) \leq \ord(F)$.  Write $E = E_\gamma$ and $F= E_\beta$. We may assume $F$ is not a line bundle, so we let $\beta = \varepsilon((p+1)/2^q)$ with $p$ even and $q\geq 1$.  Let $i\in \{0,2\}$ be such that $i\equiv p \pmod 4$, and define the slopes $\zeta_i$ and $\alpha$ as above.  Since $\ord(E)\leq \ord(F)$ and $\gamma < \beta$, we must have $\gamma \leq \alpha$.  Now apply $\Hom(E_\gamma,-)$ to the exact sequence $$0\to E_{\zeta_i}\to E_\alpha \te \Hom(E_\alpha,E_\beta) \to E_\beta\to 0.$$ By our induction hypothesis, $\Ext^i(E_\gamma,E_\alpha) = 0$ for $i>0$.  We also have $\Ext^2(E_\gamma,E_{\zeta_i}) = 0$ by Serre duality and stability since  $\zeta_i - \gamma > (\beta-3) - \gamma >-3$.   Therefore $\Ext^i(E_\gamma,E_\beta) = 0$ for $i>0$.

\emph{Case 2:} $\ord(E)\leq \ord(F)$.  In this case we instead write $E=E_\beta$ and $F=E_\gamma$.  Write $\beta = \varepsilon((p+1)/2^q)$, let $i\in \{0,2\}$ be such that $i\equiv p \pmod 4$.  Applying $\Hom(-,E_\gamma)$ to the sequence $$0\to E_\beta \to E_\eta \te \Hom(E_\beta,E_\eta)^*\to E_{\omega_i}\to 0$$ completes the proof by a similar argument to the previous case.
\end{proof}

The next result is a straightforward consequence of Theorem \ref{thm-exceptionalHom}, stability, and Serre duality.

\begin{corollary}\label{cor-exceptionalTensor}
If $E$ and $F$ are exceptional bundles, then  $E\te F$ has at most one nonzero cohomology group.  It can be determined as follows:
\begin{enumerate}
\item If $\mu(E\te F)\geq 0$, then $h^0(E\te F) = \chi(E\te F)$ and all other cohomology is zero.
\item If $-3<\mu(E\te F)<0$, then $h^1(E\te F) = -\chi(E\te F)$ and all other cohomology is zero.
\item If $\mu(E\te F) \leq -3$, then $h^2(E\te F) = \chi(E\te F)$ and all other cohomology is zero.
\end{enumerate}
\end{corollary}

The inductive description of exceptional bundles also allows us to study when a sheaf $\sHom(E,F)$ is globally generated.

\begin{proposition}\label{prop-excgg}
The exceptional bundle $E$ is globally generated if and only if $\mu(E)\geq 0$.
\end{proposition}
\begin{proof}
Since $E$ is stable, if it has a section then $\mu(E)\geq 0$.

Conversely suppose $E=E_\beta$ with $\beta\geq 0$; we induct on the order of $E_\beta$.  If $E_\beta$ is a line bundle the result is clear, so suppose $\beta = \varepsilon((p+1)/2^q)$ with $p$ even and $q\geq 1$.  Putting $\alpha = \varepsilon(p/2^q)$, we have a surjection $$E_\alpha\te \Hom(E_\alpha,E_\beta) \to E_\beta \to 0.$$  Since $\beta> 0$ we have $\alpha\geq 0$, so $E_\alpha$ is globally generated by induction.  Since $E_\beta$ is a quotient of a globally generated bundle it is also globally generated.
\end{proof}

\begin{remark}
By a similar argument, an exceptional bundle $E$ is ample if and only if $\mu(E) \geq 1$.  Indeed, if $\mu(E) \geq 1$ then $E$ is a quotient of a direct sum of copies of an ample exceptional bundle.  Conversely, if $\mu(E) < 1$ then the restriction of $E$ to any line is not ample, so $E$ is not ample.
\end{remark}

\begin{theorem}\label{thm-excgg}
Let $E$ and $F$ be exceptional bundles.  The vector bundle $\sHom(E,F)$ is globally generated if and only if there is an integer $a$ with $\mu(E)\leq a\leq \mu(F)$.  Equivalently, if $E$ and $F$ are not both the same line bundle, then $\sHom(E,F)$ is globally generated if and only if $\mu(F)-\mu(E)>\sqrt{5}-2\approx 0.237.$
\end{theorem}
\begin{proof}
By Proposition \ref{prop-excgg} we may assume that neither $E$ nor $F$ is a line bundle. Suppose there is such an integer $a$.  Then $E=E_\beta$ with $\beta = \varepsilon((p+1)/2^q)$ for $p$ even and $q\geq 1$.  Letting $\eta = \varepsilon((p+2)/2^q)$, we have an injection $$0\to E_\beta \to E_\eta \te \Hom(E_\beta,E_\eta)^*$$ and a surjection $$\sHom(E_\eta,F)\te \Hom(E_\beta,E_\eta)\to \sHom(E_\beta,F)\to 0.$$ By our assumption on $\beta$, we have $\eta \leq a$. By induction on the order,  $\sHom(E_\eta,F)$ is globally generated.  Therefore $\sHom(E_\beta,F)$ is a quotient of a globally generated bundle,  so is globally generated.

Conversely suppose that $\sHom(E,F)$ is globally generated.  Clearly $\mu(E)\leq\mu(F)$.  Considering the slopes of exceptional bundles, we see that if $\mu(E)$ and $\mu(F)$ are not the same integer then there is no integer $a$ with $\mu(E)\leq a\leq \mu(F)$ if and only if we have  $\mu(F)-\mu(E)<\sqrt{5}-2$.  Since $E$ and $F$ are not line bundles, both $\Delta(E)$ and $\Delta(F)$ are at least $3/8$.  Then we compute \begin{align*}\chi(E,F) &= r(E)r(F)(P(\mu(F)-\mu(E))-\Delta(E)-\Delta(F))\\
&<r(E)r(F)(P(\sqrt{5}-2)-\frac{3}{4})\\
&\approx 0.63 \cdot r(E)r(F).\\
&<r(E)r(F).\end{align*}
Then by Theorem \ref{thm-exceptionalHom} we have $\Hom(E,F)=\chi(E,F)<r(E)r(F)$,  so $\sHom(E,F)$ does not have enough sections to be globally generated.
\end{proof}

By combining Theorem \ref{thm-excgg} and resolutions by exceptional collections, we can give a criterion for $\sHom(W,V)$ to be globally generated for general stable bundles $W$ and $V$.

\begin{theorem}\label{thm-gghom}
Let $\bv,\bw\in K(\P^2)$ be Chern characters of stable bundles with discriminant greater than $1/2$.  Let $E_{\nu^+}$ be the primary corresponding exceptional bundle to $\bv$ and let $E_{\omega^-}$ be the secondary corresponding exceptional bundle to $\bw$.  Let $V\in M(\bv)$ and $W\in M(\bw)$ be general bundles.  If $\nu^+-\omega^- \leq 2$, then $\sHom(W,V)$ is globally generated.
\end{theorem}
\begin{proof}
Let $\nu^+ = \alpha.\beta$ be the primary corresponding exceptional slope to $\bv$.  Then following \S\S\ref{sss-res1}--\ref{sss-res2}, $V$ admits one of the resolutions $$0 \to E^{m_1}_{-\alpha-3} \to E^{m_2}_{-\beta} \oplus E^{m_3}_{-\nu^+}\to V\to 0 $$
$$0 \to E^{m_3}_{-\nu^+-3} \oplus E_{-\alpha-3}^{m_1} \to E_{-\beta}^{m_2} \to V\to 0.$$  To see that $\sHom(W,V)$ is globally generated, it is enough to show that $\sHom(W,E_{-\beta})$ and $\sHom(W,E_{-\nu^+})$ are globally generated.

Let $\omega^-= \gamma.\delta$ be the secondary corresponding exceptional slope to $\bw$.  Then $-\omega^- = (-\delta).(-\gamma)$ is the primary corresponding exceptional slope to the Serre dual $W^D$.  Thus $W^*=W^D(3)$ admits one of the resolutions
$$0 \to E^{n_1}_{\delta} \to E^{n_2}_{\gamma+3} \oplus E^{n_3}_{\omega^-+3}\to W^*\to 0 $$
$$0 \to E^{n_3}_{\omega^-} \oplus E_{\delta}^{n_1} \to E_{\gamma+3}^{n_2} \to W^*\to 0.$$
Tensoring by either $E_{-\beta}$ or $E_{-\nu^+}$, we see that it is enough to prove the four bundles $$\sHom(E_\beta,E_{\gamma+3}), \quad \sHom(E_\beta,E_{\omega^-+3}),\quad \sHom(E_{\nu^+},E_{\gamma+3}),\quad \sHom(E_{\nu^+},E_{\omega^-+3})$$ are globally generated. 

By Theorem \ref{thm-excgg}, if $\sHom(E_\beta,E_{\gamma+3})$ is globally generated then so are all the other bundles.  Thus it suffices to show there is an integer $a$ with $\beta \leq a \leq \gamma+3$.  We are assuming $\nu^+ \leq \omega^-+2$.  There is an integer $a$ with $\omega^-+2 < a \leq \gamma+3$, and for this integer we have $\beta\leq a \leq \gamma+3$.
\end{proof}

The previous result was reformulated for tensor products in the introduction in Theorem \ref{thm-ggIntro}.  In our notation here it reads as follows.
 
 \begin{corollary}
 Let $\bv,\bw\in K(\P^2)$ be Chern characters of stable bundles with discriminant greater than $1/2$.  Let $\nu^+$ be the primary corresponding exceptional slope to $\bv$ and let $\omega^+$ be the primary corresponding exceptional slope to $\bw$.  Let $V\in M(\bv)$ and $W\in M(\bw)$ be general bundles.  If $\nu^++\omega^+\leq -1$, then $V\te W$ is globally generated.
 \end{corollary}
 \begin{proof}
 Apply Theorem \ref{thm-gghom} to $\sHom(W^*,V)$.  The secondary corresponding orthogonal bundle to $W^*$ is $E_{-\omega^+-3}$.
 \end{proof}

\subsection{Twists by exceptional bundles}\label{ss-excTwist}  In this section we prove the following theorem.

\begin{theorem}\label{thm-twistExceptional}
Let $\bv\in K(\P^2)$ be the Chern character of a stable bundle, and let $V\in M(\bv)$ be general.  If $E$ is an exceptional bundle then $V\te E$ has at most one nonzero cohomology group.  
\end{theorem}

Consequently, the cohomology of $V\te E$ is  determined by the slope and the Euler characteristic:

\begin{enumerate}
\item If $\chi(V\te E) = 0$, then $V\te E$ has no cohomology.
\item If $\chi(V\te E) < 0$, then $h^1(V\te E) = -\chi(V\te E)$.
\item If $\chi(V\te E) > 0$ and $\mu(V\te E) \geq 0$, then $h^0(V\te E) = \chi(V\te E)$.
\item If $\chi(V\te E) > 0$ and $\mu(V\te E) \leq -3$, then $h^2(V\te E) = \chi(V\te E)$.
\end{enumerate}

The proof of the theorem will occupy the rest of the section.  It is enough to prove that either $H^0(V\te E) = 0$ or $H^1(V\te E) = 0$.  Indeed, if $V^D$ is the Serre dual and we know $V^D \te E^*$ has either $H^0(V^D\te E^*) = 0$ or $H^1(V^D\te E^*) = 0$, then either $H^1(V \te E)=0$ or $H^2(V\te E) = 0$.  By stability $H^0(V\te E)$ and $H^2(V\te E)$ cannot both be nonzero, so it follows that $V\te E$ has at most one nonzero cohomology group. 

By Corollary \ref{cor-exceptionalTensor} we may assume $M(\bv)$ is positive dimensional. Let $E_{\nu^+}$ be the primary corresponding exceptional bundle to $\bv$ and decompose $\nu^+$ as $\nu^+=\gamma.\delta$.   Then by \S\S\ref{sss-res1}--\ref{sss-res2}, $V$ fits in a triangle of the form
\begin{equation}\label{res}\begin{cases}E_{-\nu^+}^{m_3} \to V\to K \to \cdot & \textrm{if } \chi(V\te E_{\nu^+})> 0
\\ E_{-\nu^+-3}^{m_3}\to K\to V\to \cdot & \textrm{if } \chi(V\te E_{\nu^+}) \leq 0,\end{cases}\end{equation} where $K$ is a complex $$K:E_{-\gamma-3}^{m_1} \to E_{-\delta}^{m_2}$$ sitting in degrees $-1$ and $0$.  In the derived category, $K$ fits in the triangle $$E_{-\gamma-3}^{m_1} \to E_{-\delta}^{m_2} \to K \to \cdot.$$ The complex $K$ corresponds to a general stable Kronecker $\Hom(E_{-\gamma-3},E_{-\delta})^*$-module. 

Let $E_\beta$ be any exceptional bundle.  The argument we use to compute the cohomology of $V\te E_\beta$ depends upon the relative position of $\beta$ and the slopes $\gamma < \nu^+ < \delta.$  Recall from Proposition \ref{prop-EulerExceptional} that the sign of the Euler characteristic $\chi(V\te E_\beta)$ changes from negative to positive as $\beta$ crosses $\nu^+$.

\begin{lemma}\label{lem-betaEexceptional}
The bundle $V\te E_{\nu^+}$ has at most one nonzero cohomology group.  It is either $H^0$ or $H^1$, determined by $\chi(V\te E_{\nu^+})$.
\end{lemma}
\begin{proof}
This follows immediately from (\ref{res}) and orthogonality properties of exceptional bundles.
\end{proof}

\begin{lemma}\label{lem-betaGdelta}
Suppose $\beta \geq \delta$.  Then the only nonzero cohomology group of $V\te E_\beta$ is $H^0$.
\end{lemma}
\begin{proof}
Tensor the triangle (\ref{res}) by $E_\beta$ and compute the cohomology of the terms.  The bundle $E_\beta \te E_{-\gamma-3}$ has slope greater than $-3$, so can only have $H^0$ and $H^1$.  By Corollary \ref{cor-exceptionalTensor}, the bundle $E_\beta\te E_{-\delta}$ has only $H^0$.  Thus $E_\beta\te K$ can only have $H^{-1}$ and $H^0$.  

\emph{Case 1: $\chi(V\te E_{\nu^+})\geq 0$.}  By Corollary \ref{cor-exceptionalTensor}, the bundle $E_\beta\te E_{-\nu^+}$ can only have $H^0$.  Then, since $V$ is a sheaf, this implies $V\te E_\beta$ can only have $H^0$.

\emph{Case 2: $\chi(V\te E_{\nu^+})\leq 0$.}  The bundle $E_\beta\te E_{-\nu^+-3}$ has slope greater than $-3$, so it can only have $H^0$ or $H^1$.  Then it again follows that $V\te E_\beta$ can only have $H^0$.

In either case, by Proposition \ref{prop-EulerExceptional} we have  $\chi(V\te E_\beta)>0$, so in fact $H^0(V\te E_\beta)$ is  nonzero.
\end{proof}

\begin{lemma}\label{lem-betaMiddle}
Suppose $\beta \leq \gamma$.  Then $H^0(V\te E_\beta)=0$. 
\end{lemma}
\begin{proof}
Tensor (\ref{res}) by $E_\beta$.  Observe $E_\beta \te E_{-\gamma-3}$ has slope less than $-3$, so only has $H^2$.  Since $E_\beta \te E_{-\delta}$ has negative slope, it can only have $H^1$ and $H^2$.   Therefore $H^0(E_\beta \te K)=0.$  On the other hand clearly $H^0(E_\beta \te E_{-\nu^+})=0$ since the slope is negative, and $H^1(E_{\beta} \te E_{-\nu^+-3})=0$ since the slope is less than $-3$.  We conclude that $H^0(V\te E_\beta)=0$.
\end{proof}

The hardest exceptional bundles $E_\beta$ to handle are the ones closest to the corresponding exceptional slope.  We now turn to these cases.

\begin{lemma}\label{lem-betaSandwichRight}
Suppose $\nu^+ < \beta < \delta$.  Then the only nonzero cohomology of $V\te E_\beta$ is $H^0$.
\end{lemma}
\begin{proof} 
Let us inductively decompose the exceptional bundle $E_\beta$ as in \S\ref{ss-inductExceptional}.  Write $\beta = \alpha.\eta$ and let $\zeta_i$ be as in \S\ref{ss-inductExceptional}, so that we have an exact sequence $$0\to E_{\zeta_i}\to E_\alpha\te \Hom(E_\alpha,E_\beta)\to E_\beta \to 0.$$  The bundle $V\te E_{\zeta_i}$ has no $H^2$ since its slope is greater than $-3$:  $$\mu(V\te E_{\zeta_i}) = \mu(V) + \zeta_i > (-\nu^+-x_{\nu^+})+(\beta-3)>(-\nu^+-x_{\nu^+})+(\nu^++x_{\nu^+}-3)>-3.$$ Therefore it would be sufficient to show that $V\te E_\alpha$ has only $H^0$.

If $\chi(V\te E_{\nu^+})\geq 0$ then this strategy works perfectly: we have $\nu^+\leq \alpha < \beta < \delta$, and since $V\te E_{\nu^+}$ has only $H^0$ by Lemma \ref{lem-betaEexceptional}, we can proceed by induction on the order to show that $V\te E_\alpha$ has only $H^0$.

On the other hand, if $\chi(V\te E_{\nu^+}) < 0$, then we know that $H^1(V\te E_{\nu^+})$ is nonzero, and so we cannot use the case $\alpha = \nu^+$ as a base for our induction.  Instead, if the slope $\beta$ decomposes as $\beta = \alpha.\eta$ with $\alpha = \nu^+$, then it must be one of the slopes
\begin{align*}\epsilon_1 &= \nu^+.\delta\\
\epsilon_2 &= \nu^+.(\nu^+.\delta) = \nu^+.\epsilon_1\\
&\vdots\\
\epsilon_{j+1} &= \nu^+.\epsilon_{j}\\
&\vdots
\end{align*}
We instead show directly that if $j\geq 1$ then $V\te E_{\epsilon_j}$ has only $H^0$ to establish these base cases.

Fix some $j\geq 1$.  There are exponents $p_1$ and $p_2$ such that $E_{\epsilon_j}$ is the kernel of a general map $$0\to E_{\epsilon_j}\to E_{\delta}^{p_2}\to E_{\gamma+3}^{p_1}\to 0.$$  This sequence is essentially the Beilinson spectral sequence for $E_{\epsilon_j}$ with respect to the exceptional collection $(E_{\nu^+},E_\delta,E_{\gamma+3})$; the dual exceptional collection is $(E_{-\nu^+-3},E_{-(\gamma.\nu^+)-3},E_{-\gamma-3})$, and we have $$
\chi(E_{\epsilon_j}\te E_{-\nu^+-3})= \chi(E_{\nu^++3},E_{\epsilon_j})=\chi(E_{\epsilon_j},E_{\nu^+}) = 0
$$
and
$$p_1 = -\chi(E_{\epsilon_j}\te E_{-\gamma-3}) = -\chi(E_{\gamma+3},E_{\epsilon_j}) = -\chi(E_{\epsilon_j},E_\gamma)> 0$$
$$p_2 = -\chi(E_{\epsilon_j}\te E_{-(\gamma.\nu^+)-3}) = -\chi(E_{\gamma.\nu^++3},E_{\epsilon_j})=-\chi(E_{\epsilon_j},E_{\gamma.\nu^+})> 0,$$
 with the inequalities coming from slope considerations and Corollary \ref{cor-exceptionalTensor}.

Dually, we find that $E_{-\epsilon_j}$ is quasi-isomorphic to a general complex $$K':E_{-\gamma-3}^{p_1}\to E_{-\delta}^{p_2}$$ sitting in degrees $-1$ and $0$.  Now we use the triangle (\ref{res}) for $V$ and apply $\Hom(K',-)$.   Observe that $\Ext^i(K',E_{-\nu^+-3})=0$ for all $i$ since the bundles $E_{-\nu^+-3},E_{-\gamma-3},E_{-\delta}$ form a strong exceptional collection.  Therefore $$H^i(V\te E_{\epsilon_j})\cong \Ext^i(K',V) \cong \Ext^i(K',K)$$ for all $i$.  Finally $\Ext^i(K',K)$ can be computed by passing to Kronecker modules as in \S\ref{sssec-complexes}, and Theorem \ref{thm-kronecker} shows there is at most one nonzero group.  From Proposition \ref{prop-EulerExceptional} we have $\chi(V\te E_{\epsilon_j}) > 0$.  Therefore the nonzero group is $H^0$.
\end{proof}

\begin{lemma}\label{lem-betaSandwichLeft}
Suppose $\gamma < \beta < \nu^+$.  Then the only nonzero cohomology group of $V\te E_\beta$ is $H^1$.
\end{lemma}
\begin{proof}
The proof is largely dual to the proof of Lemma \ref{lem-betaSandwichRight}; we sketch the key differences.  Writing $\beta = \alpha.\eta$ and letting $\omega_i$ be as in \S\ref{ss-inductExceptional}, we have an exact sequence $$0\to E_\beta \to E_\eta \te \Hom(E_\beta,E_\eta)^* \to E_{\omega_i}\to 0.$$ Since $\mu(V)> -\nu^+-x_{\nu^+}$, it is clear that $\mu(V\te E_\beta) > -3$ and $H^2(V\te E_{\beta})=0$.  Thus the only issue is to show that $H^0(V\te E_\beta)=0$.  By induction on the order it would be enough to show that $H^0(V\te E_\eta)=0$.  If $\chi(V\te E_{\nu^+})\leq 0$ then this approach easily works. On the other hand, if $\chi(V\te E_{\nu^+})>0$, then we will eventually arrive at a slope that decomposes as $\beta = \alpha .\eta$, where $\eta = \gamma.\delta$.  Such slopes form a sequence given by $\epsilon_1 = \gamma.\nu^+$ and $\epsilon_{j+1} = \epsilon_j .\nu^+$, and we must show $H^1(V\te E_{\epsilon_j}) = 0$.

The bundle $E_{\epsilon_j}$ can be fit as a general cokernel $$0\to E_{\delta-3}^{p_2}\to E_\gamma^{p_1} \to E_{\epsilon_j}\to 0.$$ This resolution can be obtained by computing the Beilinson spectral sequence for $E_{\epsilon_j}$ using the full exceptional collection $(E_{\delta-3},E_\gamma,E_{\nu^+})$ with corresponding dual collection $(E_{-\delta},E_{-(\nu^+.\delta)} ,E_{-\nu^+})$.  The shifted Serre dual $E_{-\epsilon_j-3}[1]$ is then quasi-isomorphic to a general complex $$K':E_{-\gamma-3}^{p_1}\to E_{-\delta}^{p_2}$$ sitting in degrees $-1$ and $0$.  Then $$H^i(V\te E_{\epsilon_j})\cong \Ext^i(E_{-\epsilon_j},V)\cong \Ext^{2-i}(V,E_{-\epsilon_j-3})^*\cong \Ext^{1-i}(V,K')^*.$$  But $\Ext^i(E_{-\nu^+},K)=0$ for all $i$, so $\Ext^{1-i}(V,K')\cong \Ext^{1-i}(K,K')$.  This space can be computed using Kronecker modules and Theorem \ref{thm-kronecker} shows that at most one of the groups $H^i(V\te E_{\epsilon_j})$ is nonzero.  This time $\chi(V\te E_{\epsilon_j})<0$ by Proposition \ref{prop-EulerExceptional}, so the nonzero group is $H^1$.
\end{proof}

This completes the proof of Theorem \ref{thm-twistExceptional}.

\section{Overview of the main theorem}\label{sec-overview}

In this section we fix notation for the proof of the main theorem and outline the strategy of the proof.  The proof will then occupy the rest of the paper.

Let $\bv,\bw\in K(\P^2)$ be stable Chern characters with discriminant larger than $1/2$, and let $V\in M(\bv)$ and $W\in M(\bw)$ be general stable bundles.  We will view the character $\bv$ as fixed and the character $\bw$ as variable.  First we discuss our main way of writing down $V$ and construct some additional Chern characters related to $\bv$.  Then we will analyze different possibilities for the character $\bw$ and compute the cohomology of $V\te W$ depending on the relative positions of $\bv$ and $\bw$. 

\subsection{Resolutions and characters computed from $\bv$} Let $E_{\nu^+}$ be the primary corresponding exceptional bundle to $\bv$.  We summarize the discussion in \S\S\ref{sss-res1}--\ref{sss-res2}.  Decompose $\nu^+$ as $\nu^+ = \alpha.\beta$, where $\alpha = \varepsilon(p/2^q)$ and $\beta = \varepsilon((p+1)/2^q)$.  Then, according to the sign of $\chi(V \te E_{\nu^+})$, we get the following way of decomposing $V$.  
\begin{enumerate}
\item If $\chi(V \te E_{\nu^+}) > 0$, then $V$ fits in a triangle $$E_{-\nu^+}^{m_3} \to V\to K\to \cdot$$ where $K$ is a two-term complex $$K:E_{-\alpha-3}^{m_1}\to E_{-\beta}^{m_2}$$ sitting in degrees $-1$ and $0$.

\item If $\chi(V\te E_{\nu^+})\leq 0$, then $V$ has a resolution $$0\to E_{-\nu^+-3}^{m_3}\to K\to V\to 0,$$ where $K$ is a sheaf with resolution $$0\to E_{-\alpha-3}^{m_1}\to E_{-\beta}^{m_2} \to K \to 0.$$ (We can also interpret $K$ as a two-term complex for a more uniform treatment.)
\end{enumerate}

The complex $K$ corresponds to a general stable Kronecker $\Hom(E_{-\alpha-3},E_{-\beta})^*$-module of dimension vector $(m_1,m_2)$.  

There is also a particularly important character $\bu^+$ which is orthogonal to $\bv$.  It is defined up to scale by requiring it to be orthogonal to all the terms in the above decomposition of $V$.

\begin{definition}\label{def-corrOrth}
We define $\bu^+$, a \emph{primary corresponding orthogonal character to $\bv$}, up to scale, as follows.
 \begin{enumerate}
 \item If $\chi(V\te E_{\nu^+}) > 0$, then $\bu^+$ is an integral character of positive rank that satisfies \begin{align*}\chi(\bv \te \bu^+) &= 0,\\\chi(E_{-\nu^+} \te \bu^+) &= 0.\end{align*}
\item If $\chi(V\te E_{\nu^+}) \leq 0$, then $\bu^+$ is an integral character of positive rank that satisfies 
\begin{align*}
\chi(\bv \te \bu^+) &=0,\\
\chi(E_{-\nu^+-3}\te \bu^+) &= 0.
\end{align*}
 \end{enumerate}
 \end{definition}
 
 \begin{remark}
From the decomposition of $V$ we  also find $\chi(\bu^+ \te K) = 0$.  It follows that the orthogonal parabola to $K$ is the parabola passing from $E_{\nu^+}$ to $\bu^+$.
\end{remark}

\begin{remark}\label{rmk-corrOrth}
The character $\bu^+$ is always stable and $\Delta(\bu^{+})>1/2$ (see \cite[Proposition 3.7]{CHW}).
\end{remark}

\begin{remark}\label{rem-rankslope}
If $\chi(\bv \te E_{\nu^+}) > 0$, then a straightforward computation shows that the inequality $\rk(K)\geq 0$ is equivalent to the inequality $\chi(\bv \te (1,\nu^+,\frac{1}{2}+\frac{1}{2r_{\nu^+}}))\leq 0$.  The point $(\nu^+,\frac{1}{2}+\frac{1}{2r_{\nu^+}})$ lies on $E_{-\nu^+}^\perp$ at the peak of the Dr\'ezet--Le Potier curve over $I_{\nu^+}$. Since $\bu^+$ is defined by intersecting $\bv^{\perp}$ and $E_{-\nu^+}^\perp$,  we see that the inequality $\rk(K) \geq 0$ is equivalent to $\mu(\bu^+) \geq \nu^+$.
\end{remark}

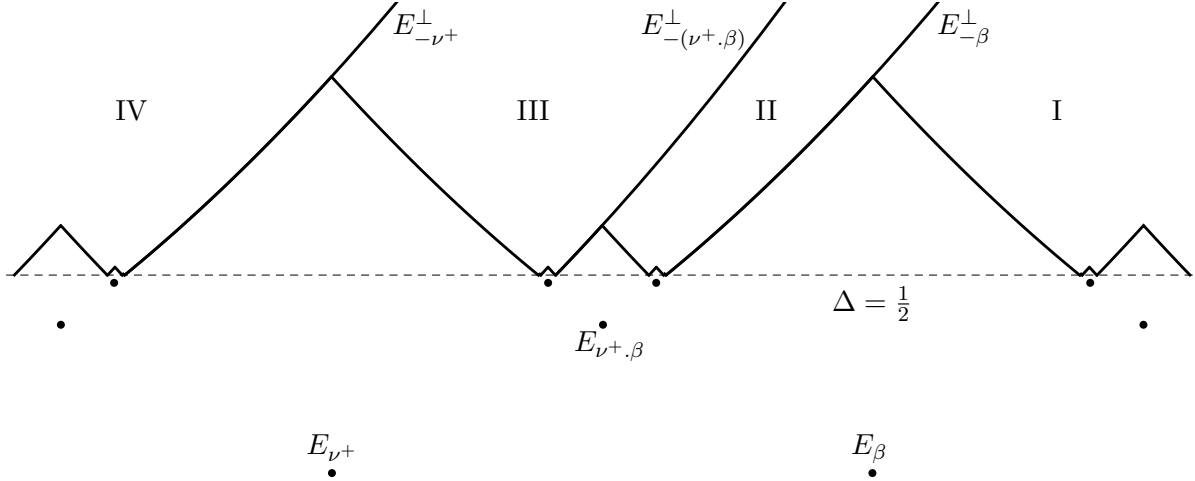
\begin{figure}[t] 
\begin{center}
\setlength{\unitlength}{1in}
\begin{picture}(6.5,2.62)
\put(0,0){\includegraphics[scale=.608,bb=0 0 10.69in 4.31in]{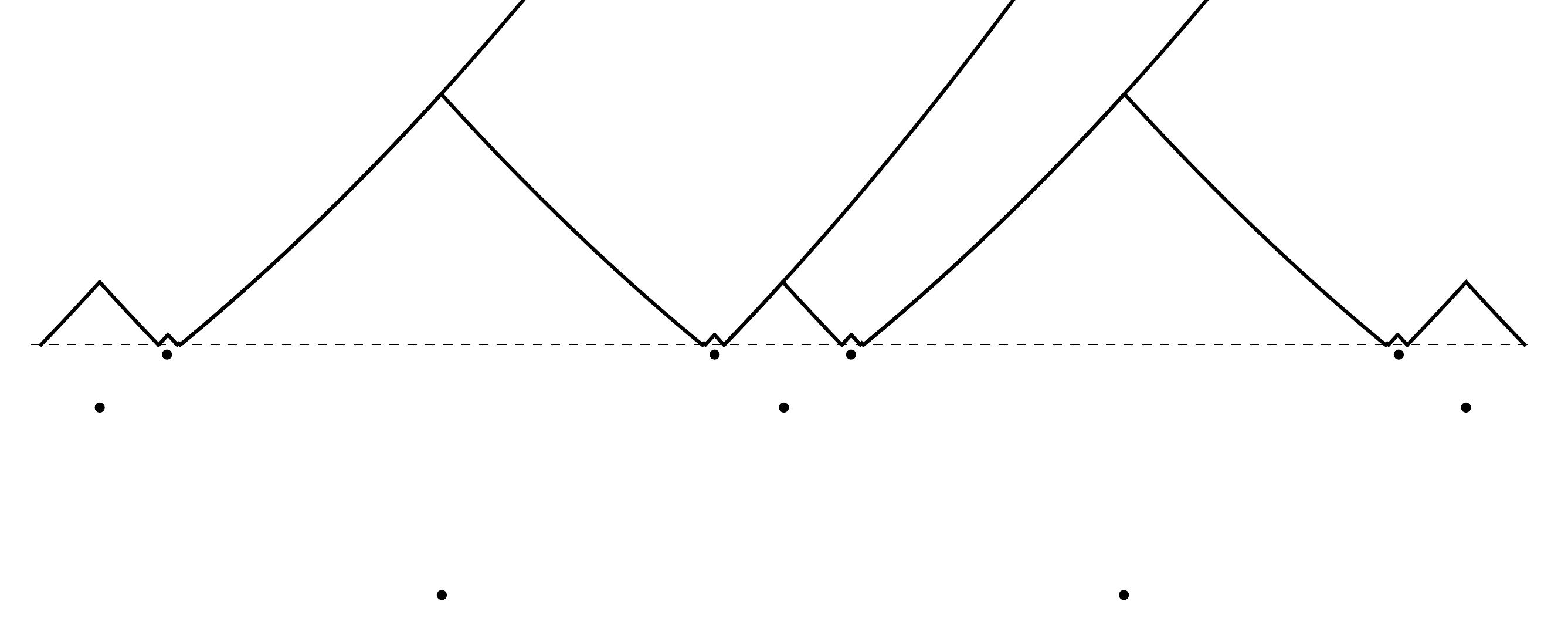}}
\put(1.7,.25){$E_{\nu^+}$}
\put(3.1,.79){$E_{\nu^+.\beta}$}
\put(4.55,.25){$E_{\beta}$}
\put(2.15,2.45){$E_{-\nu^+}^\perp$}
\put(3.45,2.45){$E_{-(\nu^+.\beta)}^\perp$}
\put(5,2.45){$E_{-\beta}^\perp$}

\put(4.45,1.){$\Delta=\frac{1}{2}$}
\put(5.6,2){I}
\put(4.05,2){II}
\put(2.8,2){III}
\put(.7,2){IV}
\end{picture}
\end{center}
\caption{The regions (I)-(IV) in the $(\mu,\Delta)$-plane.  The picture is to scale when $E_{\nu^+} = \OO_{\P^2}$.}\label{fig-4regions}
\end{figure}  

\subsection{Decomposition of the $(\mu,\Delta)$-plane} Let $E_{\omega^+}$ be the primary corresponding exceptional bundle to $\bw$.  The strategy we use to compute the cohomology of $V\te W$ depends on the position of $(\mu(\bw),\Delta(\bw))$ in the $(\mu,\Delta)$-plane. We impose conditions on $\omega^+$ and various Euler characteristics to restrict $(\mu(\bw),\Delta(\bw))$ to several regions; by using Proposition \ref{prop-corrExc}
we can see how an inequality, e.g., $\omega^+ \leq -\beta$, describes a region in the $(\mu,\Delta)$-plane.
We define four regions that cover the portion of the $(\mu,\Delta)$-plane above the Dr\'ezet--Le Potier curve.  
We depict these regions in Figure \ref{fig-4regions}.


\begin{itemize}
\item[(I)] $\omega^+ \leq -\beta$  and  $\chi(W\te E_{-\beta}) \geq 0$.
\item[(II)] $-\beta \leq \omega^+ \leq -(\nu^+.\beta) $  and $\chi(W\te E_{-\beta}) \leq 0$  and $\chi(W\te E_{-(\nu^+.\beta)}) \geq 0$.
\item[(III)]  $-(\nu^+.\beta) \leq \omega^+ \leq -\nu^+$ and $\chi(W\te E_{-(\nu^+.\beta)}) \leq 0$ and $\chi(W\te E_{-\nu^+}) \geq 0$.
\item[(IV)] $-\nu^+ \leq \omega^+$ and if  $-\nu^+=\omega^+$ then $\chi(W\te E_{-\nu^+}) \leq 0$.
\end{itemize}

\begin{warning}Figure \ref{fig-4regions} is slightly misleading in some boundary cases.  Note that the characters in Figure \ref{fig-4regions} lying on the left branch of the Dr\'ezet--Le Potier curve over $I_{\nu^+}$ are in both regions (III) and (IV).
\end{warning}

\subsubsection{Regions (I) and (II)}  When $\bw$ is in region (I) or (II), the tensor product $V\te W$ is quite ``positive.''  Correspondingly, we will show $H^1(V\te W) = 0$.  See Propositions \ref{prop-regionI} and \ref{prop-regionII}.

\subsubsection{Region (III)}
Things are more challenging in region (III), owing to the fact that the orthogonal parabola to $\bv$ crosses region (III).  Therefore, the type of cohomology that $V\te W$ has  depends on the position of $\bw$ within region (III).  The computation of the cohomology of $V\te W$ depends directly on the sign of $\chi(K\te W)$.  We therefore subdivide region (III) into two regions (IIIa) and (IIIb):
\begin{enumerate}
\item[(IIIa)] $\bw$ is in region (III) and $\chi(K\te W) \geq 0$.
\item[(IIIb)] $\bw$ is in region (III) and $\chi(K\te W) \leq 0$.
\end{enumerate}

When $\bw$ is in region (IIIa), we again show that $H^1(V\te W) = 0$ (see Propositions \ref{prop-RIII1} (1) and \ref{prop-RIII2}).  If $\chi(V\te E_{\nu^+}) > 0$ and $\bw$ is in region (IIIb), then we show that $V\te W$ is usually special in Proposition \ref{prop-RIII1} (2).  On the other hand if $\chi(V\te E_{\nu^+}) \leq 0$ then region (IIIb) is not so important; we define a new region (V) to be the union of region (IV) and region (IIIb), and will handle this new region uniformly.

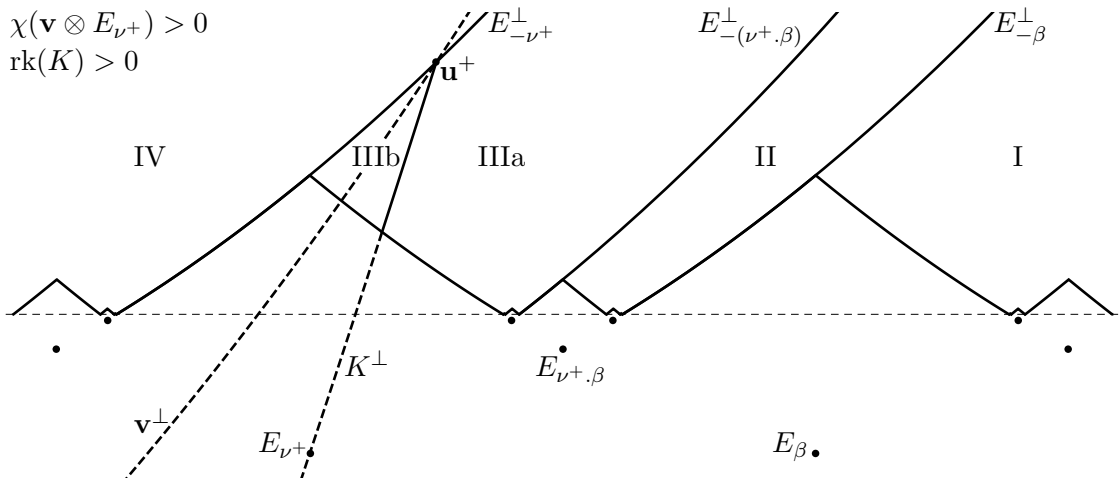
\begin{figure}[p] 
\begin{center}
\setlength{\unitlength}{1in}
\begin{picture}(6.5,7.98)
\put(0,0){\includegraphics[scale=.57,bb=0 0 11.42in 14in]{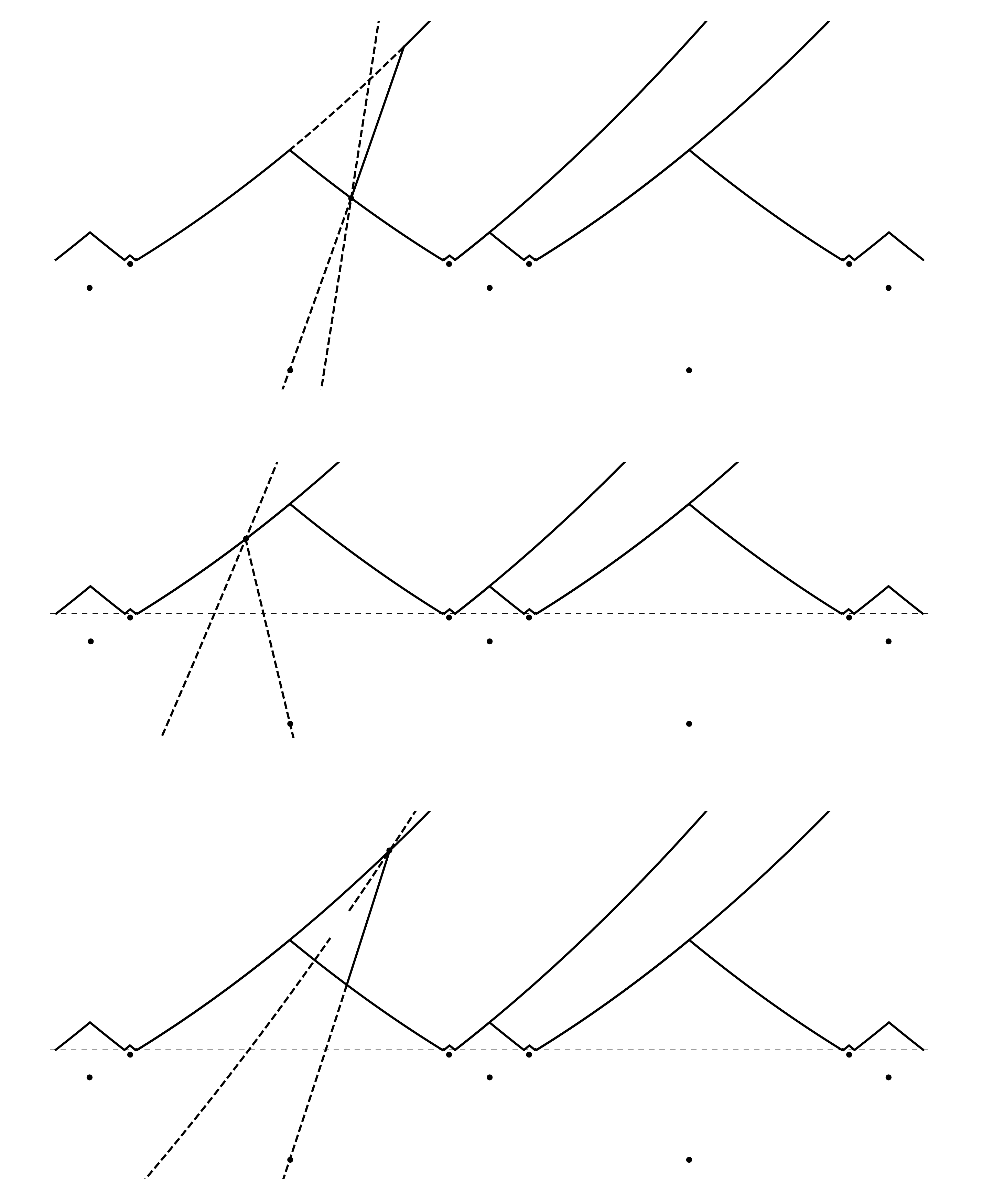}}
\put(.1,2.5){$(3)$}
\put(.35,2.5){$\chi(\bv \te E_{\nu^+})>0$}
\put(.35,2.3){$\rk(K)>0$}

\put(1.65,.3){$E_{\nu^+}$}
\put(4.34,.3){$E_{\beta}$}
\put(3.1,.7){$E_{\nu^+.\beta}$}
\put(1,.4){$\bv^\perp$}
\put(2.6,2.25){$\bu^+$}
\put(2.1,.7){$K^\perp$}
\put(2.85,2.5){$E_{-\nu^+}^\perp$}
\put(3.95,2.5){$E_{-(\nu^+.\beta)}^\perp$}
\put(5.5,2.5){$E_{-\beta}^\perp$}
\put(1,1.8){IV}
\put(5.6,1.8){I}
\put(4.25,1.8){II}
\put(2.8,1.8){IIIa}
\put(2.14,1.8){IIIb}

\put(.1,4.8){$(2)$}
\put(.35,4.8){$\chi(\bv \te E_{\nu^+})>0$}
\put(.35,4.6){$\rk(K)< 0$}
\put(1.65,3.2){$E_{\nu^+}$}
\put(4.34,3.2){$E_{\beta}$}
\put(3.1,3.59){$E_{\nu^+.\beta}$}
\put(1,3.35){$\bv^\perp$}
\put(1.43,4.4){$\bu^+$}
\put(1.85,3.6){$K^\perp$}
\put(2.25,4.8){$E_{-\nu^+}^\perp$}
\put(3.4,4.8){$E_{-(\nu^+.\beta)}^\perp$}
\put(4.9,4.8){$E_{-\beta}^\perp$}
\put(1,4.3){IV}
\put(5.6,4.5){I}
\put(4,4.5){II}
\put(2.8,4.5){IIIa}

\put(.1,7.75){$(1)$ $\chi(\bv\te E_{\nu^+})\leq 0$}
\put(1.65,5.55){$E_{\nu^+}$}
\put(4.34,5.55){$E_{\beta}$}
\put(3.1,5.93){$E_{\nu^+.\beta}$}
\put(2.18,5.5){$\bv^\perp$}
\put(2.1,6.55){$\bu^+$}
\put(2.47,6.95){$K^\perp$}
\put(2.85,7.75){$E_{-\nu^+}^\perp$}
\put(3.95,7.75){$E_{-(\nu^+.\beta)}^\perp$}
\put(5.5,7.75){$E_{-\beta}^\perp$}
\put(2.53,6.55){$E_{-\nu^+-3}^\perp$}
\put(1,7.05){V}
\put(5.6,7.05){I}
\put(4.25,7.05){II}
\put(2.8,7.05){IIIa}
\end{picture}
\end{center}
\caption{The refined regions partitioning the $(\mu,\Delta)$-plane above the Dr\'ezet--Le Potier curve in the three main cases of Theorem \ref{thm-main1}.  Note that dotted lines do not subdivide regions. See Example \ref{ex-refined}.}\label{fig-4Rrefined}
\end{figure}

\begin{example}\label{ex-refined}
In Figure \ref{fig-4Rrefined} we display the refined regions in the three main cases of Theorem \ref{thm-main1}.
\begin{enumerate}
\item $\chi(\bv\te E_{\nu^+})\leq 0$.  We have defined region (V) to be the union of region (IV) and region (IIIb).  The character $\bu^+$ is at the intersection of $\bv^\perp$ and $E_{-\nu^+-3}^\perp$ and lies on the Dr\'ezet--Le Potier curve.
\item $\chi(\bv\te E_{\nu^+}) >0$ and $\rk(K) < 0$.  In this case region (IIIb) has empty interior, so no special cohomology arises.  By Remark \ref{rem-rankslope} we have $\mu(\bu^+)\leq \nu^+$, and $\bu^+$ lies on the Dr\'ezet--Le Potier curve.  Since $\rk(K)< 0$, region (IIIa) lies \emph{above} the parabola $K^\perp$.

\item $\chi(\bv \te E_{\nu^+}) > 0$ and $\rk(K) >0$.   Here $\mu(\bu^+) > \nu^+$ and region (IIIb) has nonempty interior.  For $\bw$ in the interior of region (IIIb), the cohomology of $V\te W$ is special.
\end{enumerate}
 The figures in Figure \ref{fig-4Rrefined} are to scale for the characters $(1)$ $\bv = (1,10,67)$, (2) $\bv = (3,3,26/3)$, and (3) $\bv = (4,1,9/4)$, each with $E_{\nu^+} = \OO_{\P^2}$.
\end{example}

\subsubsection{The orthogonal parabola and region (IV) or (V)}  Up to this point it has not been necessary to assume the character $\bw$ is sufficiently divisible.  For the rest of the argument we use substantially different methods and this hypothesis is crucial.  

In region (IV) or (V) it is most important to compute $H^i(V\te W)$  when $\chi(V\te W) = 0$.  As a starting point, we have that $\bu^+$ lies in region (IIIa), so if $U^+\in M(\bu^+)$ is general, then $V\te U^+$ has no cohomology.  For $\bw$ in this region and orthogonal to $\bv$, we have $\mu(\bw) \geq \mu(\bu^+)$.  In Theorem \ref{thm-orthogonal} we then show $V\te W$ has no cohomology.  Finally, in \S\ref{sec-interpolation} we  explain how to use a handful of tricks and our previous results to show that $V\te W$ is nonspecial for all other characters in the region.  This will complete the proof of Theorem \ref{thm-main1}.

\section{The first three regions}\label{sec-regions}

Here we use the notation and assumptions from \S\ref{sec-overview} and compute the cohomology of the general tensor product $V\te W$ when $\bw$ lies in regions (I), (II), or (IIIa).  We also compute the cohomology when $\chi(V\te E_{\nu^+})\geq 0$ and $\bw$ lies in region (IIIb).  This is the only case where the cohomology of $V\te W$ can be special.

\subsection{Region (I)} Recall that region (I) is defined by the inequalities $$\omega^+ \leq -\beta \quad \textrm{and} \quad \chi(W\te E_{-\beta}) \geq 0.$$ When $\bw$ lies in this region, we use Theorem \ref{thm-twistExceptional} to compute the cohomology of $V\te W$.

\begin{proposition}\label{prop-regionI}
Suppose $\bw$ lies in region (I).  Then $H^i(V\te W) = 0$ for $i>0$.
\end{proposition}

\begin{proof}
First suppose $\chi(V\te E_{\nu^+}) > 0$.  Then we have the triangle $$E_{-\nu^+}^{m_3} \to V \to K \to \cdot $$ where $K$ fits in a triangle $$E_{-\alpha - 3}^{m_1} \to E_{-\beta}^{m_2} \to K\to \cdot.$$  The assumption that $\bw$ is in region (I) shows that $W\te E_{-\beta}$ and $W\te E_{-\nu^+}$ have only $H^0$ by Theorem \ref{thm-twistExceptional}.  On the other hand $W\te E_{-\alpha-3}$ has no $H^2$ since $-\alpha-3 > \omega^-$ by Remark \ref{rem-corrExcGap}.  Therefore $W\te K$ can only have $H^{-1}$ and $H^0$, and $V\te W$ has only $H^0$.

Next suppose $\chi(V\te E_{\nu^+}) \leq 0$.  Now we have the exact sequences $$0\to E_{-\nu^+-3}^{m_3} \to K \to V \to 0$$$$0\to E_{-\alpha-3}^{m_1} \to E_{-\beta}^{m_2} \to K \to 0.$$  Here $W\te K$ can only have $H^0$, and this time the bundle $W\te E_{-\nu^+-3}$ has no $H^2$ by Remark \ref{rem-corrExcGap}.   Therefore $V\te W$ has only $H^0$.
\end{proof}

\subsection{Region (II)} We defined region (II) by the inequalities 
$$-\beta \leq \omega^+ \leq -(\nu^+.\beta)\quad\textrm{and}\quad\chi(W\te E_{-\beta}) \leq 0\quad\textrm{and}\quad\chi(W\te E_{-(\nu^+.\beta)}) \geq 0.$$ When $\bw$ lies in region (II), we use a Beilinson spectral sequence for $W$ and Theorem \ref{thm-twistExceptional} to compute the cohomology of $V\te W$.

\begin{proposition}\label{prop-regionII}
Suppose $\bw$ lies in region (II).   Then $H^i(V\te W) = 0$ for $i>0$.
\end{proposition}
\begin{proof}
In this case, we write down the Beilinson spectral sequence for $W$ with respect to the exceptional collection $(E_{\beta-3},E_\alpha,E_{\nu^+})$.  The dual exceptional collection is $(E_{-\beta},E_{-(\nu^+.\beta)},E_{-\nu^+})$.  We have
 \begin{align*}
 \chi(W\te E_{-\beta}) & \leq 0\\
 \chi(W\te E_{-(\nu^+.\beta)}) & \geq 0\\
 \chi(W\te E_{-\nu^+}) & \geq 0
 \end{align*} since $\bw$ lies in region (II), so by Theorem \ref{thm-twistExceptional} the $E_1$-page of the spectral sequence for $W$ takes the shape
$$\xymatrix{
E_{\beta-3}^{p_3} \ar[r] & 0  \ar[r] & 0  \\
0\ar[r] & E_{\alpha}^{p_1}  \ar[r] & E_{\nu^+}^{p_2}.  \\
}$$
Since the spectral sequence converges to the sheaf $W$ in degree $0$, the bottom map must be injective with some cokernel $K'$: $$0 \to E_\alpha^{p_1} \to E_{\nu^+}^{p_2} \to K' \to 0,$$ and we get a resolution of $W$ of the form $$0\to E_{\beta-3}^{p_3} \to K' \to W\to 0.$$

Now we distinguish two cases according to the sign of $\chi(V\te E_{\nu^+})$.  First suppose $\chi(V\te E_{\nu^+})\geq 0$.  Then $V\te E_\alpha$ has only $H^1$ and $V\te E_{\nu^+}$ has only $H^0$, so $V\te K'$ has only $H^0$. We know $V\te E_{\beta-3}$ has only $H^1$ by Remark \ref{rem-corrExcGap}, so $V\te W$ has only $H^0$.

Next suppose $\chi(V\te E_{\nu^+}) < 0$.  In this case we write down the standard resolution of $V$
$$0\to E_{-\nu^+-3}^{m_3} \to K \to V \to 0$$$$0\to E_{-\alpha-3}^{m_1} \to E_{-\beta}^{m_2} \to K \to 0.$$ We find that $W \te E_{-\nu^+-3}$ has no $H^2$, so it suffices to show that $W\te K$ has only $H^0$.  However, $K\in M(\ch K)$ is a general stable bundle in its moduli space, and $\chi( K \te E_{\nu^+}) = 0$.  It easily follows that $E_{\nu^+}$ is the primary corresponding exceptional bundle for $\ch K$.  The region (II) only depends on $\nu^+$ and not on the character $\bv$, so $\bw$ still lies in region (II) if we replace $\bv$ by $\ch K$.  But then applying the previous paragraph to $K$ we see that $W\te K$ has only $H^0$.
\end{proof}

\subsection{Region (III)}\label{ssec-RIII} 
Finally, suppose $\bw$ lies in region (III), defined by the inequalities
$$-(\nu^+.\beta) \leq \omega^+ \leq -\nu^+\quad \textrm{and}  \quad \chi(W\te E_{-(\nu^+.\beta)}) \leq 0 \quad\textrm{and}\quad \chi(W\te E_{-\nu^+}) \geq 0.$$The Beilinson spectral sequence for $W$ with respect to the exceptional collection $(E_{\beta-3},E_\alpha,E_{\nu^+})$ (with dual collection $(E_{-\beta},E_{-(\nu^+.\beta)},E_{-\nu^+})$) shows that $W$ fits into a triangle $$E_{\nu^+}^{n_3} \to W \to K'\to \cdot$$ where $K'$ fits into a triangle $$E_{\beta-3}^{n_2}\to E_{\alpha}^{n_1} \to K' \to \cdot.$$  The same discussion in \S\ref{sss-res1} that shows the complex $K$ for $V$ is given by a general map also shows that $K'$ is given by a general map; however it need \emph{not} correspond to a stable Kronecker module (the moduli space of semistable modules of dimension vector $(n_2,n_1)$ may be empty).  Our work on Kronecker modules gives us the following key to the computation.

\begin{lemma}\label{lem-derivedTensor}
The derived tensor product $K\te K'$ has only $H^0$ or $H^1$, determined by $\chi(K\te K')$.
\end{lemma}
\begin{proof}
Notice that $(K')^*(-3)[1]$ (with the dual being derived) is a general complex of the form $$E_{-\alpha-3}^{n_1}\to E_{-\beta}^{n_2}$$ sitting in degrees $-1$ and $0$.  But we compute $$H^i(K\te K') = \Ext^i((K')^*,K) = \Ext^{2-i}(K,(K')^*(-3))^* = \Ext^{1-i}(K,(K')^*(-3)[1])^*.$$ Since $K$ corresponds to a semistable Kronecker module, this can be computed using Kronecker modules and Theorem \ref{thm-kronecker}.
\end{proof}

\begin{remark}
Since $\chi(E_{\alpha.\beta}\te K) = 0$, we have $\chi(K\te K') = \chi(K \te W)$.  
\end{remark}

Recall that subregions of region (III) were defined by

\begin{enumerate}
\item[(IIIa)] $\bw$ is in region (III) and $\chi(K\te W) \geq 0$;
\item[(IIIb)] $\bw$ is in region (III) and $\chi(K\te W) \leq 0$.
\end{enumerate}
The next result completely computes the cohomology of $V\te W$ if $\chi(V\te E_{\nu^+})> 0$.

\begin{proposition}\label{prop-RIII1}
Suppose $\bw$ is in region (III) and $\chi(V\te E_{\nu^+})> 0$.  
\begin{enumerate}
\item If $\bw$ is in region (IIIa), then $H^i(V\te W)=0$ for $i>0$.
\item If $\bw$ is in region (IIIb), then \begin{align*}h^0(V\te W) &= \chi(\bv\te E_{\nu^+})\chi(\bw\te E_{-\nu^+}) \\ h^1(V\te W) &= -\chi(K\te W) \\ h^2(V\te W) &= 0.\end{align*}
\end{enumerate}
In particular, $V\te W$ is special if $\chi(K\te W) < 0$ and $\chi(\bv\te E_{\nu^+})$ and $\chi(\bw\te E_{-\nu^+})$ are both positive.
\end{proposition}
\begin{proof}
Since $\chi(V\te E_{\nu^+})\geq 0$, we have the triangles $$E_{-\nu^+}^{m_3} \to V\to K \to \cdot$$ $$E_{\nu^+}^{n_3} \to W\to K'\to \cdot. $$We observe that $E_{\nu^+} \te K$ and $E_{-\nu^+}\te K'$ both have no cohomology in any degree.  Therefore the cohomology of $W \te K$ is the same as the cohomology of $K\te K'$ (which is computed by Lemma \ref{lem-derivedTensor}), and the cohomology of $W\te E_{-\nu^+}^{m_3}$ is the same as the cohomology of $E_{\nu^+}^{n_3}\te E_{-\nu^+}^{m_3}$, which just has $$h^0(E_{\nu^+}^{n_3}\te E_{-\nu^+}^{m_3}) = m_3n_3 = \chi(\bv \te E_{\nu^+})\chi(\bw\te E_{-\nu^+}).$$  Tensoring the triangle for $V$ by $W$, the result follows at once.
\end{proof}

On the other hand, when $\chi(V\te E_{\nu^+})\leq 0$ we can compute the cohomology of $V\te W$ if $\bw$ is in region (IIIa).

\begin{proposition}\label{prop-RIII2}
Suppose $\bw$ is in region (IIIa) and $\chi(V\te E_{\nu^+}) \leq 0.$ Then $H^i(V\te W) = 0$ for $i>0$.
\end{proposition}
\begin{proof}
This time we have triangles 
$$0\to E_{-\nu^+-3}^{m_3}\to K\to V\to 0$$ $$E_{\nu^+}^{n_3} \to W\to K'\to \cdot.$$
Again, $W\te K$ has the same cohomology as $K\te K'$, thus has only $H^0$ by Lemma \ref{lem-derivedTensor}.  Our assumptions on $W$ give $\chi(W\te E_{-(\alpha.\beta)-3}) \leq 0$, so $W\te E_{-(\alpha.\beta)-3}$ has only $H^1$ and $V\te W$ can only have $H^0$.
\end{proof}

\section{Cohomologically orthogonal bundles}\label{sec-orthogonal}

In this section we study the cohomology of a general tensor product $V\te W$ when $\chi(V\te W) = 0$.  Recall we defined $\bu^+$, a primary orthogonal character to $\bv$, in Definition \ref{def-corrOrth}.  We repeat the defining property here for convenience: 
\begin{enumerate} 
 \item If $\chi(V\te E_{\nu^+}) > 0$, then $\bu^+$ is an integral character of positive rank which satisfies \begin{align*}\chi(\bv \te \bu^+) &= 0\\\chi(E_{-\nu^+} \te \bu^+) &= 0.\end{align*}
\item If $\chi(V\te E_{\nu^+}) \leq 0$, then $\bu^+$ is an integral character of positive rank which satisfies 
\begin{align*}
\chi(\bv \te \bu^+) &=0\\
\chi(E_{-\nu^+-3}\te \bu^+) &= 0.
\end{align*}
 \end{enumerate}
Here we introduce another important character which is orthogonal to $\bv$.

\begin{definition}\label{secondary-orth}
We define $\bu^-$, a \emph{secondary corresponding orthogonal character to $\bv$}, to be the \emph{dual} of a primary corresponding orthogonal character \emph{of rank at least $2$} to the Serre dual $\bv^D$.
\end{definition}

\begin{remark}
By Serre duality, $\chi(\bv\te \bu^-) = \chi(\bv^D \te (\bu^-)^*) = 0$.  Thus $\bu^-$ is in fact an orthogonal character to $\bv$.  Since $\bu^-$ has rank at least $2$, the general $U^-\in M(\bu^-)$ is a vector bundle.  Thus the cohomology of $V\te U^-$ can be analyzed using Serre duality.
\end{remark}

We note the following basic fact about the characters $\bu^{\pm}$.  

\begin{lemma}\label{lem-kroneckerCorrExc}
Let $\bv$ be a character with corresponding exceptional bundles $E_{\nu^{\pm}}$ and corresponding orthogonal characters $\bu^{\pm}$.

\begin{enumerate}
\item For the character $\bu^+$, the primary corresponding exceptional bundle is $E_{-\nu^+}$.  

\item For the character $\bu^-$, the secondary corresponding exceptional bundle is $E_{-\nu^--3}$.
\end{enumerate}
\end{lemma}
\begin{proof}
This follows immediately from Remark \ref{rmk-corrOrth} and Proposition \ref{prop-corrExc}.
\end{proof}

Our starting point for cohomological vanishing is the next theorem.  This is a slight generalization of \cite[Theorem 7.1]{CHW}; however, we have essentially already reproved it here in \S\ref{sec-regions}.

\begin{theorem}\label{thm-extremalVanish}
Let $\bv\in K(\P^2)$ be a stable Chern character with $\Delta(\bv) > 1/2$, and let $\bu^\pm$ be the corresponding orthogonal characters.  Let $V\in M(\bv)$ and $U^{\pm}\in M(\bu^\pm)$ be general.  Then $V\te U^{\pm}$ has no cohomology.
\end{theorem}
\begin{proof}
The primary corresponding exceptional bundle to $\bu^+$ is $E_{-\nu^+}$ by Lemma \ref{lem-kroneckerCorrExc}.  We also have $\chi(K\te \bu^+) = 0$ from the decomposition of $V$.  Thus $\bu^+$ lies in region (IIIa).  Either Proposition \ref{prop-RIII1} (1) or Proposition \ref{prop-RIII2} shows $V\te U^+$ has no cohomology.  

The case of $\bu^-$ follows by Serre duality.
\end{proof}

Now we combine the bundles $U^\pm$ to produce additional cohomologically orthogonal bundles.

\begin{theorem}\label{thm-orthogonal}
Suppose $\chi(\bv\te \bw) = 0$ and $\bw$ is sufficiently divisible.  Suppose  either
\begin{enumerate}
\item $\mu(\bw)\geq \mu(\bu^+)$ or
\item $\mu(\bw)\leq \mu(\bu^-)$.
\end{enumerate}
Then the general tensor product $V\te W$ has no cohomology.
 \end{theorem}
 \begin{remark}
If $M(\bw)$ is positive dimensional, then the converse is true: if $\mu(\bu^-)<\mu(\bw)<\mu(\bu^+)$ then the general tensor product $V\te W$ has nontrivial cohomology.  This essentially follows from Proposition \ref{prop-RIII1} (2).
 \end{remark}
 
 \begin{proof}[Proof of Theorem \ref{thm-orthogonal}]
By Serre duality, we may as well focus on the case $\mu(\bw) \geq \mu(\bu^+)$.  Furthermore, we may assume equality does not hold, for if $\mu(\bw) = \mu(\bu^+)$, then $\bw$ is a primary corresponding orthogonal character and Theorem \ref{thm-extremalVanish} would apply.  The characters $\bu^\pm$ form a $\QQ$-basis for $\bv^\perp \te \QQ \subset K(\P^2) \te \QQ$, so we can find integers $m>0$ and $m_+,m_-$ with $$m \bw = m_+ \bu^+ + m_- \bu^-.$$ The coefficients $m_{\pm}$ cannot both be negative since $\bw$ has positive rank.  If they were both positive then $\mu(\bw)$ would be a weighted mean of $\mu(\bu^-)$ and $\mu(\bu^+)$, hence it would lie between them.  Thus $m_+$ and $m_-$ must have different signs.  If $m_+<0$ then we find $\mu(\bu^-)$ is a weighted mean of $\mu(\bw)$ and $\mu(\bu^+)$, which is again a contradiction.  We conclude that $m_+ > 0$ and $m_- < 0$.
 
  Replace $\bw$ with $m \bw$, replace $\bu^+$ with $m_+\bu^+$, and replace $\bu^-$ with $-m_- \bu^-$, so that we now have $$\bw = \bu^+-\bu^-.$$ Let $U^\pm \in M(\bu^\pm)$ be general bundles.  Let $\phi\in \Hom(U^-,U^+)$ be a general homomorphism.  We will show that $\phi$ is injective and the cokernel $W$ given by the sequence $$0\to U^- \to U^+ \to W\to 0$$ is a prioritary sheaf (see \S\ref{ss-prioritary}) of character $\bw$ such that $V\te W$ has no cohomology.  Since the stack of semistable sheaves is an open substack of the irreducible stack of prioritary sheaves, this implies that if $W\in M(\bw)$ is a general semistable sheaf then $V\te W$ has no cohomology.

We prove the sheaf $\sHom(U^-,U^+)$ is globally generated as an application of Theorem \ref{thm-gghom}.  By Lemma \ref{lem-kroneckerCorrExc}, the secondary corresponding exceptional bundle to $U^-$ is $E_{-\nu^--3}$, and the primary corresponding exceptional bundle to $U^+$ is $E_{-\nu^+}$.  Then $$(-\nu^+)- (-\nu^--3)=\nu^- - \nu^+ + 3\leq -3+3=0$$ by Remark \ref{rem-corrExcGap}, and Theorem \ref{thm-gghom} shows $\sHom(U^-,U^+)$ is globally generated.
Then by a Bertini-type theorem \cite[Proposition 2.6]{HuizengaJAG}, the map $\phi$ is injective and the cokernel $W$ is a vector bundle since $r(W) \geq 2$.   Clearly $V\te W$ has no cohomology by Theorem \ref{thm-extremalVanish} since $V\te U^-$ and $V\te U^+$ have no cohomology.  

To show that $W$ is prioritary we need to show that $\Ext^2(W,W(-1)) = 0$.  Applying $\Hom(W,-)$ to the sequence $$0\to U^-(-1)\to U^+(-1) \to W(-1)\to 0,$$ it is enough to show $\Ext^2(W,U^+(-1)) = 0$.  Applying $\Hom(-,U^+(-1))$ to $$0\to U^- \to U^+ \to W\to 0$$ gives an exact sequence $$\Ext^1(U^-,U^+(-1))\to \Ext^2(W,U^+(-1))\to \Ext^2(U^+,U^+(-1)).$$ Since $\Ext^2(U^+,U^+(-1)) =0$ by stability, we are reduced to proving $\Ext^1(U^-,U^+(-1)) = 0$.  For this we need to show $H^1((U^-)^* \te U^+(-1))=0$, which we claim follows from an application of Proposition \ref{prop-regionI}.

The primary corresponding exceptional slope to $U^+(-1)$ is $-\nu^++1$ by Lemma \ref{lem-kroneckerCorrExc}.  The primary corresponding exceptional slope to $(U^-)^D$ is the dual of the secondary corresponding exceptional slope to $U^-$, so is $\nu^- + 3$.  Therefore the primary corresponding exceptional slope to $(U^-)^* = (U^-)^D(3)$ is $\nu^-$.  A straightforward computation shows that if we take characters $\bv = \ch(U^+(-1))$ and $\bw = \ch((U^-)^*)$, then $\bw$ lies in region (I) determined by $\bv$.  Then Proposition \ref{prop-regionI} completes the proof.
\end{proof}

\begin{remark}
We can see how divisible $\bw$ must be for the proof to work.  The characters $\bu^\pm$, which are determined by $\bv$, span a sublattice $$\Lambda = \Z\bu^+ \oplus \Z\bu^- \subset \bv^\perp \cong \Z^2.$$  We can find an integer $m$ such that $m\bv^\perp \subset \Lambda$.  Then any character $\bw\in \bv^\perp$ which is divisible by $m$ will work.   Note that due to the fractal-like nature of the Dr\'ezet--Le Potier curve, the characters $\bu^{\pm}$ depend on $\bv$ in a somewhat complicated way.  Heuristically, if $\bu^{\pm}$ have high rank, then we might expect $m$ to be large.
 \end{remark}

\section{The remaining cases}\label{sec-interpolation}

In this section we collect a few simple tricks for using known computations of the cohomology of general tensor products to deduce computations of  the cohomology of other general tensor products.  Together with the results in \S\S\ref{sec-exceptional}--\ref{sec-orthogonal}, this will allow us to compute the cohomology of $V\te W$ when $\bw$ lies in region (IV) (if $\chi(\bv\te E_{\nu^+})>0$) or (V) (if $\chi(\bv \te E_{\nu^+})\leq 0$), completing the proof of Theorem \ref{thm-main1}.

\subsection{Interpolation tricks}  An \emph{elementary modification} of a sheaf $W$ is a sheaf $W'$ obtained as the kernel $$0 \to W' \to W\to \OO_p\to 0$$ where $\OO_p$ is a skyscraper sheaf.  If $W$ is torsion-free then so is $W'$, and if $W$ is prioritary then so is $W'$.  If $V\te W$ has no sections, then $V\te W'$ has no sections.  See \cite[Lemma 2.7]{CoskunHuizengaBN} for details.

\begin{lemma}\label{lem-modification}
Suppose $W'$ is a prioritary sheaf of invariants $(r',\mu,\Delta')$ and $H^0(V\te W') = 0$.  Let $\bw = (r,\mu,\Delta)$ be a character with the same slope and $\Delta\geq \Delta'$.  Then after replacing $\bw$ by a sufficiently divisible multiple, the general prioritary sheaf $W$ of character $\bw$ has $H^0(V\te W) = 0$.
\end{lemma}
\begin{proof}
Write $\Delta-\Delta' = \frac{p}{q}$ with $p,q$ positive integers such that $r'q$ is a multiple of $r$.  Starting from the bundle $(W')^{\oplus q}$ of discriminant $\Delta'$, perform $pr'$ elementary modifications to construct a sheaf $W$.  Each modification increases the discriminant by $1/(qr')$, so $W$ has character $(rq,\mu,\Delta)$.  It is prioritary, it is a multiple of $\bw$, and it has the required cohomology vanishing.  
\end{proof}

If we have two stable bundles $W$ and $W'$ with sufficiently close slopes such that $V\te W$ and $V\te W'$ have no sections, then we can combine them to create such bundles with new slopes.

\begin{lemma}\label{lem-aveSlope}
Suppose $W'$ and $W''$ are stable sheaves of invariants $\bw' = (r',\mu',\Delta')$ and $\bw'' = (r'',\mu'',\Delta'')$ and that $H^0(V\te W') = 0$ and $H^0(V\te W'')=0$.  Suppose $0 < \mu'' - \mu' < 2$.  For any slope $\mu$ with $\mu'< \mu < \mu''$, there is a direct sum $W = (W')^{\oplus a} \oplus (W'')^{\oplus b}$ of slope $\mu$.  Then $W$ is prioritary, $H^0(V\te W) = 0$, and $\Delta(W) < \max\{\Delta',\Delta''\}.$
\end{lemma}
\begin{proof}
The slope of a direct sum $(W')^{\oplus a} \oplus (W'')^{\oplus b}$ is a weighted mean of $\mu(W')$ and $\mu(W'')$, and it is easy to arrange the slope to be any rational number $\mu$ between $\mu'$ and $\mu''$.  Prioritariness of $W$ follows from the stability of $W'$ and $W''$ and the assumption $0< \mu''-\mu' < 2$.  The cohomology of $V\te W$ is immediate.  Let $\bu$ be a character orthogonal to both $W'$ and $W''$.  Then the three points $(\mu',\Delta'),(\mu,\Delta(W)),(\mu'',\Delta'')$ all lie on the orthogonal parabola to $\bu$, and $\Delta(W) < \max\{\Delta',\Delta''\}$ follows from convexity.
\end{proof}

If we have bundles $W$ and $W'$ of the same slope such that $V\te W$ and $V\te W'$ are both nonspecial with cohomology in the same degree, then we can combine them to handle intermediate discriminants.

\begin{lemma}\label{lem-aveDisc}
Let $W'$ and $W''$ be stable sheaves of invariants $\bw' = (r',\mu,\Delta')$ and $\bw'' = (r'',\mu,\Delta'')$ with $\Delta'\leq \Delta'$.  Suppose there is an index $j$ such that $H^i(V\te W') = 0$ and $H^i(V\te W'') = 0$ for $i\neq j$.  Let $\bw = (r,\mu,\Delta)$ be a character with $\Delta' \leq \Delta \leq \Delta''$.  Then after replacing $\bw$ by a sufficiently divisible multiple, the general prioritary sheaf $W$ of character $\bw$ has $H^i(V\te W) = 0$ for $i\neq j$.
\end{lemma}
\begin{proof}
The discriminant of $W=(W')^{\oplus a}\oplus (W'')^{\oplus b}$ is a weighted mean of $\Delta'$ and $\Delta''$, so by choosing the exponents appropriately we can ensure it has discriminant $\Delta$.  Prioritariness follows as in Lemma \ref{lem-aveSlope}, and the cohomology of $V\te W$ is evident.
\end{proof}

\subsection{Region (IV)} In this section assume $\chi(V\te E_{\nu^+})>0$.  Then region (IV) was defined by the condition that $-\nu^+ \leq \omega^+$, and if $-\nu^+=\omega^+$ then $\chi(W\te E_{-\nu^+}) \leq 0$.  We can further subdivide region (IV) into the following regions.
\begin{enumerate}
\item[(IVa)] $\chi(V\te W) \geq 0$ and $\chi(W\te E_{-\nu^+}) \leq 0$.
\item[(IVb)] $\mu(W) \geq \mu(\bu^+)$ and $\chi(V\te W) \leq 0$.
\item[(IVc)] $\nu^+ - x_{\nu^+} < \mu(W)\leq \mu(\bu^+)$ and $\chi(W\te E_{-\nu^+}) \leq 0$.
\item[(IVd)] $\mu(W)<\nu^+-x_{\nu^+}$.
\end{enumerate}
We compute the cohomology of $V\te W$ in each subregion.

\begin{proposition}\label{prop-restReg}
Suppose $\chi(V\te E_{\nu^+}) > 0$.  Replace $\bw$ by a sufficiently large multiple.
\begin{enumerate}
\item If $\bw$ is in region (IVa) then $H^1(V\te W) = 0$.
\item If $\bw$ is in regions (IVb), (IVc), or (IVd), then $H^0(V\te W) = 0$.
\end{enumerate}
\end{proposition}
\begin{proof}
(IVa) In this case the character $\bw$ lies directly above a character $\bw'$ in region (IIIa) and directly below a sufficiently divisible character $\bw''$ on the orthogonal parabola to $\bv$.  Let $W'\in M(\bw')$ and $W''\in M(\bw'')$ be general.  Then $H^1(V\te W')=0$ by Proposition \ref{prop-RIII1} (1) and $H^1(V\te W'')=0$ by Theorem \ref{thm-orthogonal}. After replacing $\bw$ by a sufficiently divisible multiple we can use Lemma \ref{lem-aveDisc} to show $H^1(V\te W) = 0$.

(IVb) Characters in region (IVb) lie above the orthogonal parabola.  Take a sufficiently divisible character $\bw'$ on the orthogonal parabola to $\bv$, and let $W'\in M(\bw')$ be general.  By Theorem \ref{thm-orthogonal} we have $H^0(V\te W') = 0$.  Then replacing $\bw$ by a sufficiently divisible multiple, Lemma \ref{lem-modification} shows $H^0(V\te W) = 0$.
 
 (IVc) In this case $\bw$ lies directly above a character $\bw'$ on the orthogonal parabola to $E_{-\nu^+}$, but this time $\bw'$ is in region (IIIb) so we use Proposition \ref{prop-RIII1} (2).  In the notation of the proposition, we have $n_3=\chi(W'\te E_{-\nu^+})=0$, so $H^0(V\te W') = 0.$  Then after replacing $\bw$ by a sufficiently divisible multiple, Lemma \ref{lem-modification} shows $H^0(V\te W)=0$.

(IVd) If $E_\gamma$ is any exceptional bundle with $\gamma < \nu^+$, then $H^0(V\te E_\gamma) = 0$.  Given any slope $\mu$ with $\mu< \nu^+-x_{\nu^+}$ we can find two exceptional bundles $E_\gamma$ and $E_\delta$ with $\gamma \leq  \mu \leq \delta < \nu^+$ and $\delta - \gamma \leq 1$.  We use Lemma \ref{lem-aveSlope} to find a prioritary sheaf $W'$ of slope $\mu$ with discriminant less than $1/2$ such that $V\te W'$ has no sections.  Then replacing $\bw$ by a sufficiently divisible multiple we can use Lemma \ref{lem-modification} to find a prioritary sheaf $W$ of character $\bw$ such that $H^0(V\te W)=0.$
\end{proof}

\begin{figure}[t] 
\begin{center}
\setlength{\unitlength}{1in}
\begin{picture}(6.5,3.27)
\put(0,0){\includegraphics[scale=.608,bb=0 0 10.69in 5.38in]{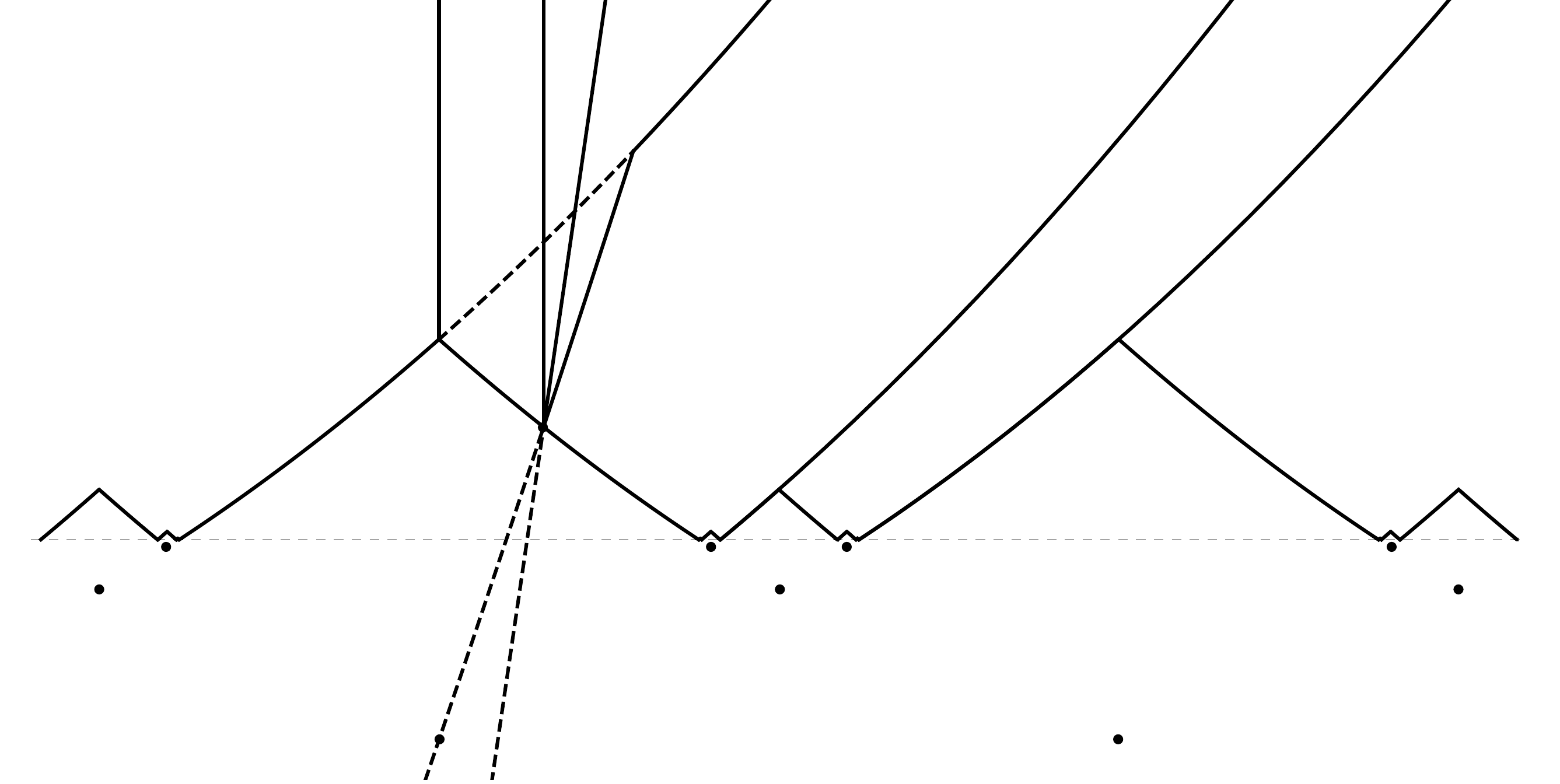}}
\put(1.55,.15){$E_{\nu^+}$}
\put(3.1,.65){$E_{\nu^+.\beta}$}
\put(4.43,.15){$E_{\beta}$}
\put(2.18,.6){$\bv^\perp$}
\put(2.02,1.4){$\bu^+$}
\put(2.47,2){$K^\perp$}
\put(3.2,3.1){$E_{-\nu^+}^\perp$}
\put(4.4,3.1){$E_{-(\nu^+.\beta)}^\perp$}
\put(6,3.1){$E_{-\beta}^\perp$}
\put(5.8,2){I}
\put(4.45,2){II}
\put(3,2){IIIa}
\put(2.65,3){Va}
\put(2.28,3){Vb}
\put(1.96,3){Vc}
\put(.8,2){Vd}
\put(2.5,1.33){$E_{-\nu^+-3}^\perp$}
\end{picture}
\end{center}
\caption{The subdivision of region (V) when  $\chi(\bv\te E_{\nu^+}) \leq 0$.  See \S\ref{ss-IVprime}.}\label{fig-IVprime}
\end{figure}

\subsection{Region (V)}\label{ss-IVprime}  Here we assume $\chi(V\te E_{\nu^+})\leq 0$.  Region (V) was defined to be the union of region (IV) and region (IIIb). We cover it by the following regions.
\begin{enumerate}
\item[(Va)] This is the region above region (IIIa) and below the orthogonal parabola to $\bv$.  It is given by the inequalities $\chi(V\te W) \geq 0$ and either $\chi(K\te W) \leq 0$ or $\chi(E_{-\nu^+}\te W)\leq 0$.
\item[(Vb)] $\mu(W) \geq \mu(\bu^+)$ and $\chi(V\te W)\leq 0$.
\item[(Vc)] $\nu^+ \leq \mu(W)\leq \mu(\bu^+)$.
\item[(Vd)] $\mu(W) \leq \nu^+$.
\end{enumerate}
See Figure \ref{fig-IVprime} for a picture of these regions.

\begin{proposition}
Suppose $\chi(V\te E_{\nu^+}) \leq 0$.  Replace $\bw$ by a sufficiently large multiple.
\begin{enumerate}
\item If $\bw$ is in region (Va) then the general tensor product $V\te W$ has no higher cohomology.
\item If $\bw$ is in regions (Vb), (Vc), or (Vd), then the general tensor product $V\te W$ has no sections.
\end{enumerate}
\end{proposition}
\begin{proof}
Regions (Va), (Vb), and (Vd) are handled by methods analogous to Proposition \ref{prop-restReg}.

(Vc) Both $V\te E_{\nu^+}$ and $V\te U^+$ have no sections, and any stable bundle with slope between $\nu^+$ and $\mu(U^+)$ has discriminant larger than $\Delta(U^+)$.  So we use Lemmas \ref{lem-aveSlope} and \ref{lem-modification} to construct the required sheaf $W$.
\end{proof}

\subsection{Torsion sheaves}\label{ssec-torsion}

In this subsection we allow $\bv$ to be the Chern character of a one-dimensional semistable sheaf, and we let $M(\bv)$ be the moduli space parameterizing (pure) semistable sheaves of character $\bv$.  We compute the cohomology of a general tensor product $V\te W$, where $W$ is a general stable vector bundle with sufficiently divisible character.  We normalize $\bv = (0,d,\chi)$ in terms of its first Chern class $d$ and Euler characteristic $\chi$.  The general point $V\in M(\bv)$ is a line bundle of Euler characteristic $\chi$ supported on a smooth curve $C$ of degree $d$.

In the $(\mu,\Delta)$-plane, the space $\bv^\perp$ becomes the vertical line $$\mu = -\frac{\chi}{d}.$$  Thus $\bv^\perp$ crosses the Dr\'ezet--Le Potier curve in a unique point.  The primary corresponding exceptional bundle $E_{\nu^+}$ is the exceptional bundle with $-\frac{\chi}{d}\in I_{\nu^+}$.  We consider the Beilinson spectral sequence of a general $V$ as in \S \ref{sss-res1} and \ref{sss-res2}.  The sequence degenerates in the same way as in the positive rank cases; following the proof of \cite[Proposition 5.3]{CHW}, the key to seeing this is to show that there is a resolution of $V$ of one of the following two forms depending on the sign of $\chi(V\te E_{\nu^+})$: $$0\to E_{-\alpha-3}^{m_1}\to E_{-\beta}^{m_2}\oplus E_{-\nu^+}^{m_3} \to V \to 0$$$$0\to E_{-\nu^+-3}^{m_3}\oplus E_{-\alpha-3}^{m_1}\to E_{-\beta}^{m_2}\to V\to 0.$$ The proof of this given in \cite[Proposition 5.3]{CHW} in the higher rank case needs a minor modification, since we cannot use the theory of prioritary sheaves.  Instead, we need to prove the semistability of a general such quotient directly.  But if $\phi\in \Hom(E_{-\alpha-3}^{m_1},E_{-\beta}^{m_2}\oplus E_{-\nu^+}^{m_3})$ is general, then by \cite[Theorem 2.8]{Ottaviani} the locus $\det \phi = 0$ can only be singular at points where $\phi$ drops rank by at least $2$, and this happens in codimension $4$.  Therefore $\det \phi = 0$ is a smooth curve $C$ of degree $d$ and $V$ is a line bundle on $C$.  Thus $V$ is semistable.

Having established the shape of the Beilinson spectral sequence, the results from \S \ref{sec-exceptional}-\S\ref{sec-regions} go through with only minor modification.  In Figure \ref{fig-4Rrefined}, the first two diagrams should be modified by replacing $\bv^\perp$ with a vertical line, but otherwise the diagrams remain accurate.  Note that the vertical line $\bv^\perp$ cannot both pass through the left branch of the Dr\'ezet-Le Potier curve and lie right of $E_{\nu^+}$, so the third case in Figure \ref{fig-4Rrefined} never occurs.

Elementary modifications are particularly useful for studying the cohomology of $V\te W$ when $V$ is torsion:  if $$0\to W' \to W \to \OO_p\to 0$$ is an elementary modification with $p$ not in the support of $V$, then $V\te W$ and $V\te W'$ have the same cohomology.  There is no need to generalize the results of \S\ref{sec-orthogonal} to this case, since elementary modification of the bundle $U^+$ will be the required sheaves.  The methods of \S\ref{sec-interpolation} will then complete the proof of the following theorem.

\begin{theorem}
Let $\bv$ be the Chern character of a one-dimensional semistable sheaf, and let $\bw$ be the Chern character of a stable bundle.  Suppose that $\bw$ is sufficiently divisible.  If $V\in M(\bv)$ and $W\in M(\bw)$ are general sheaves, then $V\te W$ has only one nonzero cohomology group.
\end{theorem}

\bibliographystyle{plain}

\begin{thebibliography}{ACGH85}


\bibitem[ABCH13]{ABCH}
D. Arcara, A. Bertram, I. Coskun and J. Huizenga, The birational geometry of the Hilbert Scheme of Points on the plane and Bridgeland stability conditions, Adv. Math., {\bf 235} (2013), 580--626.





\bibitem[CH14]{CoskunHuizengaMonomial} I. Coskun\ and\ J. Huizenga, Interpolation, Bridgeland stability and monomial schemes in the plane, J. Math. Pures Appl. (9) {\bf 102} (2014), no.~5, 930--971.

\bibitem[CH15]{CoskunHuizengaGokova} I. Coskun and J. Huizenga, The birational geometry of the moduli spaces of sheaves on $\mathbb{P}^2$, Proceedings of the G\"okova Geometry-Topology Conference 2014, (2015), 114--155.



\bibitem[CH20]{CoskunHuizengaBN} I. Coskun and J. Huizenga, Brill-Noether theorems and globally generated vector bundles on Hirzebruch surfaces, Nagoya Math. J., {\bf 238} (2020), 1--36.





\bibitem[CHW17]{CHW}
I. Coskun, J. Huizenga and Matthew Woolf, The effective cone of the moduli spaces of sheaves on the plane, J. Eur. Math. Soc., {\bf 19} no. 5 (2017), 1421--1467.




\bibitem[Dre86]{DrezetBeilinson} J.-M. Dr\'{e}zet, Fibr\'es exceptionnels et suite spectrale de Beilinson g\'en\'eralis\'ee sur ${\bf P}\sb 2({\bf C})$, Math. Ann. {\bf 275} (1986), no.~1, 25--48.

\bibitem[Dre87]{Drezet}
J.-M. Drezet, Fibr\'es exceptionnels et vari\'et\'es de modules de faisceaux semi-stables sur ${\bf P}\sb 2({\bf C})$, J. Reine Angew. Math. {\bf 380} (1987), 14--58. 




\bibitem[DLP85]{DLP}
J.-M. Dr\'{e}zet\ and\ J. Le Potier, Fibr\'es stables et fibr\'es exceptionnels sur ${\bf P}\sb 2$, Ann. Sci. \'Ecole Norm. Sup., {\bf 18} no. 2 (1985), 193--243. 





\bibitem[Gie77]{Gieseker}
D. Gieseker, On the moduli space of vector bundles on an algebraic surface, Ann. Math. {\bf 106} (1977), 45--60.


\bibitem[GHi94]{GottscheHirschowitz}
L. G\"{o}ttsche and A. Hirschowitz, Weak Brill-Noether for vector bundles on the projective plane, in Algebraic
geometry (Catania, 1993/Barcelona, 1994), 63--74, Lecture Notes in Pure and Appl. Math., 200 Dekker, New York.



 

\bibitem[HiL93]{HirschowitzLaszlo}
A. Hirschowitz\ and\ Y. Laszlo, Fibr\'es g\'en\'eriques sur le plan projectif, Math. Ann., {\bf 297} (1993), no.~1, 85--102.

\bibitem[Hui16]{HuizengaJAG} J. Huizenga, Effective divisors on the Hilbert scheme of points in the plane and interpolation for stable bundles, J. Algebraic Geom. {\bf 25} (2016), no.~1, 19--75.








\bibitem[LeP97]{LePotier}
J. Le Potier, {\it Lectures on vector bundles}, translated by A. Maciocia, Cambridge Studies in Advanced Mathematics, 54, Cambridge Univ. Press, Cambridge, 1997. 

\bibitem[LZ19]{LiZhao} C. Li and X. Zhao, Birational models of moduli spaces of coherent sheaves on the projective plane, Geom. Topol. \textbf{23} (2019), no.\ 1, 347--426.


\bibitem[Mar78]{Maruyama}
M. Maruyama, Moduli of stable sheaves II, {\em J. Math. Kyoto}, {\bf 18} (1978), 557--614.







\bibitem[Ott95]{Ottaviani}
G. Ottaviani, {\em Variet\`a proiettive di codimensione piccola}, Ist. nazion. di alta matematica F. Severi, 2, Aracne, Rome, 1995.




\bibitem[Sch91]{Schofield}
A. Schofield, Semi-invariants of quivers, J. London Math. Soc. (2), {\bf 43} (1991), no.~3, 385–-395






\end{thebibliography}

\end{document}